\documentclass[12pt]{amsart}
\usepackage{amsmath}
\usepackage{amssymb}
\usepackage[all]{xy}
\usepackage{longtable}
\usepackage[mathscr]{eucal}
\usepackage{xcolor}
\usepackage[bookmarks=false]{hyperref}
\usepackage{multirow,bigdelim}
\usepackage{nccmath}
\usepackage{url}
\usepackage{ulem}
\usepackage{comment}
\usepackage{mathrsfs}
\usepackage{arydshln}
\newtheorem{theorem}[equation]{Theorem}
\newtheorem{lemma}[equation]{Lemma}
\newtheorem{proposition}[equation]{Proposition}
\newtheorem{corollary}[equation]{Corollary}
\newtheorem{conjecture}[equation]{Conjecture}
\newtheorem{definition}[equation]{Definition}

\newtheorem{theorem-n}{Theorem}
\newtheorem{claim-n}[theorem-n]{Claim}
\newtheorem{lemma-n}[theorem-n]{Lemma}
\newtheorem{proposition-n}[theorem-n]{Proposition}

\theoremstyle{definition}
\newtheorem{definition-n}[theorem-n]{Definition}
\newtheorem{example}[equation]{Example}

\theoremstyle{remark}
\newtheorem{remark}[equation]{Remark}
\newtheorem{remark-n}[theorem-n]{Remark}

\numberwithin{equation}{subsection}

\allowdisplaybreaks[1]

\newcommand{\FF}{\mathbb{F}}
\newcommand{\ZZ}{\mathbb{Z}}
\newcommand{\QQ}{\mathbb{Q}}
\newcommand{\RR}{\mathbb{R}}
\newcommand{\LL}{\mathbb{L}}
\newcommand{\TT}{\mathbb{T}}

\newcommand{\CC}{\mathbb{C}}

\newcommand{\NN}{\mathbb{N}}

\newcommand{\Fq}{\mathbb{F}_{q}}
\newcommand{\Fp}{\mathbb{F}_{p}}

\newcommand{\ba}{\mathbf{a}}
\newcommand{\bb}{\mathbf{b}}

\newcommand{\bff}{\mathbf{f}}
\newcommand{\bg}{\mathbf{g}}

\newcommand{\bn}{\mathbf{n}}

\newcommand{\bv}{\mathbf{v}}

\newcommand{\bone}{\mathbf{1}}

\newcommand{\cB}{\mathcal{B}}

\newcommand{\cE}{\mathcal{E}}

\newcommand{\cH}{\mathcal{H}}
\newcommand{\cI}{\mathcal{I}}

\newcommand{\cP}{\mathcal{P}}

\newcommand{\cL}{\mathcal{L}}

\newcommand{\cZ}{\mathcal{Z}}

\newcommand{\IT}{\mathcal{I}^{\mathrm{T}}}
\newcommand{\ITw}{\mathcal{I}^{\mathrm{T}}_{w}}
\newcommand{\IND}{\mathcal{I}^{\mathrm{ND}}}
\newcommand{\INDz}{\mathcal{I}^{\mathrm{ND}_{0}}}
\newcommand{\INDw}{\mathcal{I}^{\mathrm{ND}}_{w}}
\newcommand{\INDzw}{\mathcal{I}^{\mathrm{ND}_{0}}_{w}}
\newcommand{\sI}{\mathscr{I}}
\newcommand{\sL}{\mathscr{L}}
\newcommand{\sLL}{\mathscr{L}^{\Li}}
\newcommand{\sLZ}{\mathscr{L}^{\zeta}}
\newcommand{\sLB}{\mathscr{L}^{\bullet}}
\newcommand{\sDB}{\mathscr{D}^{\bullet}}
\newcommand{\sz}{*^{\zeta}}
\newcommand{\sLi}{*^{\Li}}
\newcommand{\SB}{S^{\bullet}}

\DeclareMathAlphabet{\matheur}{U}{eur}{m}{n}

\newcommand{\eR}{\matheur{R}}

\newcommand{\fs}{\mathfrak{s}}

\newcommand{\fn}{\mathfrak{n}}
\newcommand{\fu}{\mathfrak{u}}

\newcommand{\sB}{\mathscr{B}}
\newcommand{\sC}{\mathscr{C}}
\newcommand{\sR}{\mathscr{R}}

\newcommand{\sU}{\mathscr{U}}
\newcommand{\sX}{\mathscr{X}}
\newcommand{\sUB}{\mathscr{U}^{\bullet}}
\newcommand{\sBC}{\mathscr{BC}}

\newcommand{\rT}{\mathrm{T}}

\DeclareMathOperator{\Ker}{Ker} \DeclareMathOperator{\GL}{GL}
\DeclareMathOperator{\Mat}{Mat}

\DeclareMathOperator{\Id}{Id} 
 \DeclareMathOperator{\wt}{wt}
\DeclareMathOperator{\Li}{Li}

\DeclareMathOperator{\dep}{dep}
\DeclareMathOperator{\Span}{Span}

\DeclareMathOperator{\Frac}{Frac}

\DeclareMathOperator{\Supp}{Supp}

\DeclareMathOperator{\Init}{Init}

\newcommand{\ok}{\overline{k}}

\newcommand{\tr}{\mathrm{tr}}

\newcommand{\tpi}{\widetilde{\pi}}

\newcommand{\power}[2]{{#1 [\![ #2 ]\!]}}
\newcommand{\laurent}[2]{{#1 (\!( #2 )\!)}}

\definecolor{ForestGreen}{rgb}{0.0, 0.5, 0.0}

\makeatletter
\newcommand{\xequal}[2][]{\ext@arrow 0055{\equalfill@}{#1}{#2}}
\def\equalfill@{\arrowfill@\Relbar\Relbar\Relbar}
\makeatother

\title [On Thakur's basis conjecture for  multiple zeta values]{On Thakur's basis conjecture for multiple zeta values in positive characteristic}

\author{Chieh-Yu Chang}
\address{Department of Mathematics, National Tsing Hua University, Hsinchu City 30042, Taiwan
  R.O.C.}
\email{cychang@math.nthu.edu.tw}

\author{Yen-Tsung Chen}
\address{Department of Mathematics, National Tsing Hua University, Hsinchu City 30042, Taiwan
  R.O.C.}
\email{ytchen.math@gmail.com}

\author{Yoshinori Mishiba}
\address{Department of Mathematical Sciences, University of the Ryukyus, 1 Senbaru, Nishihara-cho, Okinawa 903-0213,  Japan}
\email{mishiba@sci.u-ryukyu.ac.jp}

\thanks{The firs and second authors are partially supported by MOST Grant 107-2628-M-007-002-MY4. The third author was supported by JSPS KAKENHI Grant Number JP18K13398.}



\date{\today}

\begin{document}

\maketitle

\begin{abstract}
In this paper, we study multiple zeta values (abbreviated as MZV's) over function fields in positive characteristic. Our main result is to prove Thakur's basis conjecture, which plays  the analogue of Hoffman's basis conjecture for real MZV's. As a consequence, we derive Todd's dimension conjecture, which is the analogue of Zagier's dimension conjecture for classical real MZV's. 
\end{abstract}

\section{Introduction}

\subsection{Classical conjectures}  In this paper, we study  multiple zeta values (abbreviated as MZV's) over function fields in positive characteristic introduced by Thakur~\cite{T04}. Our motivation arises from Zagier's dimension conjecture and Hoffman's basis conjecture for classical real MZV's.  

The special value of the Riemann $\zeta$-function at positive integer $s\geq 2$ is the following series
\[\zeta(s):=\sum_{n=1}^{\infty} \frac{1}{n^{s}}\in \RR^{\times}.  \]
Classical real MZV's are generalizations of the special $\zeta$-values above. It was initiated by Euler on double zeta values and fully generalized by Zagier~\cite{Za94} in the 1990s. An {\it{admissible index}} is an $r$-tuple of positive integer $\fs=(s_{1},\ldots,s_{r})\in \ZZ_{>0}^{r}$ with $s_{1}\geq 2$.  The real MZV at $\fs$ is defined by the following multiple series
\[
\zeta(\fs):=\sum_{n_{1}> \cdots > n_{r}\geq 1}  \frac{1 }{n_{1}^{s_{1}} \cdots n_{r}^{s_{r}} } \in \RR^{\times}. \]
We call $\wt(\fs):=\sum_{i=1}^{r}s_{i}$ and $\dep(\fs):=r$ the weight and  depth of the presentation $\zeta(\fs)$ respectively.  Over the past decades, the study of real MZV's has attracted many researchers' attention as MZV's have many interesting and important connections with various topics. For example, MZV's occur as periods of mixed Tate motives  by Terasoma~\cite{Te02}, Goncharov~\cite{Gon02} and Deligne-Goncharov~\cite{DG05}, and MZV's of depth two have close connection with modular forms by Gangl-Kaneko-Zagier~\cite{GKZ06} etc. For more details and relevant references, we refer the reader to the  books~\cite{An04, Zh16, BGF19}. 

By the theory of regularized double shuffle relations~\cite{R02, IKZ06}, there are rich $\QQ$-linear relations among the same weight MZV's. One core problem on this topic is Zagier's following dimension conjecture: 

\begin{conjecture}[Zagier's dimension conjecture]\label{Con:Zagier} For an integer $w\geq 2$, we let $\mathfrak{Z}_w$ be the $\QQ$-vector space spanned by real  MZV's of weight $w$. We put  $d_{0}:=1$, $d_{1}:=0$, $d_{2}:=1$ and $d_{w}:=d_{w-2}+d_{w-3}$ for integers $w\geq 3$. Then for each integer $w\geq 2$, we have 
\[ {\rm{dim}}_{\QQ}\mathfrak{Z}_{w}=d_{w} . \]
\end{conjecture}

The best known result towards Zagier's dimension conjecture until now has been the {\it{upper bound}} result proved by Terasoma~\cite{Te02} and Goncharov~\cite{Gon02} independently. Namely, they showed that ${\rm{dim}}_{\QQ}\mathfrak{Z}_{w}\leq d_{w}$ for all integers $w\geq 2$. Due to numerical computation, Hoffman~\cite{Ho97} proposed the following conjectural basis for $\mathfrak{Z}_{w}$ for each $w\geq 2$.

\begin{conjecture}[Hoffman's basis conjecture]\label{Con:Hoffman} For an integer $w\geq 2$, we let $\mathcal{I}_{w}^{\rm{H}}$ be the set of admissible indices $\fs=(s_{1},\ldots,s_{r})$ with $s_{i}\in \left\{2,3 \right\}$ satisfying $\wt(\fs)=w$. Then the following set 
\[\mathcal{B}_{w}^{\rm{H}}:=\left\{ \zeta(\fs)|\fs\in \mathcal{I}_{w}^{\rm{H}}  \right\}\]
is a basis of the $\QQ$-vector space $\mathfrak{Z}_{w}$.
\end{conjecture}

We mention that Hoffman's basis conjecture implies Zagier's dimension conjecture. In~\cite{Br12}, Brown proved  Hoffman's basis conjecture for {\it{motivic}} MZV's. As there is a surjective map from motivic MZV's to real MZV's, Brown's theorem implies that Hoffman's conjectural basis would be a generating set for $\mathfrak{Z}_{w}$, and as a consequence the upper bound result of Terasoma and Goncharov would be derived. Our main results in this paper are to prove the analogues of Conjecture~\ref{Con:Zagier} and Conjecture~\ref{Con:Hoffman} in the function fields setting. 

\subsection{The main results}
In the positive characteristic setting, we let $A:=\FF_{q}[\theta]$ the polynomial ring in the variable $\theta$ over a finite field $\FF_{q}$ of $q$ elements, where $q$ is a power of a prime number $p$. We let $k$ be the field of fractions of $A$, and $|\cdot|_{\infty}$ be the normalized absolute value on $k$ at the infinite place $\infty$ of $k$ for which $|\theta|_{\infty}=q$.  Let $k_{\infty}:=\laurent{\FF_{q}}{1/\theta}$ be the completion of $k$ with respect to $|\cdot|_{\infty}$. We then fix an algebraic closure $\overline{k_{\infty}}$ of $k_{\infty}$ and still denote by $|\cdot|_{\infty}$ the extended absolute value on $\overline{k_{\infty}}$. We let $\CC_{\infty}$ be the  completion of $\overline{k_{\infty}}$ with respect to $|\cdot|_{\infty}$. Finally, we let $\ok$ be the algebraic closure of $k$ inside $\CC_{\infty}.$

Let $A_{+}$ be the set of monic polynomials of $A$, which plays an analogous role to the set of positive integers now in the function field setting. In~\cite{T04}, Thakur introduced the following positive characteristic MZV's: for any $r$-tuple of positive integers $\fs=(s_{1},\ldots,s_{r})\in \ZZ_{>0}^{r}$, 
\begin{equation}\label{E:MZV's}
    \zeta_{A}(\fs):=\sum_{a_{1},\cdots, a_{r}\in A_{+}}  \frac{1 }{a_{1}^{s_{1}} \cdots a_{r}^{s_{r}} } \in k_{\infty}
\end{equation}
with the restriction that $|a_{1}|_{\infty}>\cdots >|a_{r}|_{\infty}$. As our absolute value $|\cdot|_{\infty}$ is non-archimedean, the series $\zeta_{A}(\fs)$ converges in $k_{\infty}$. However, Thakur~\cite{T09} showed that such series are in fact non-vanishing. The weight and the  depth of the presentation $\zeta_{A}(\fs)$ are defined to be  $\wt(\fs):=\sum_{i=1}^{r}s_{i}$ and $\dep(\fs):=r$  respectively. When $\dep(\fs)=1$, these values are called Carlitz zeta values as initiated by Carlitz in~\cite{Ca35}.

As from now on we focus on MZV's in positive characteristic, in what follows MZV's will be Thakur's ($\infty$-adic) MZV's. In~\cite{T10}, Thakur established a product on MZV's, which is called $q$-shuffle relation in this paper. Namely, the product of two MZV's can be expressed as an $\FF_{p}$-linear combination of MZV's whose weights are the same.  It follows that the $k$-vector space $\cZ$ spanned by all MZV's form an algebra. 

Note that $\ok$-algebraic relations among MZV's are $\ok$-linear relations among monomials of MZV's, which can be expressed as $\ok$-linear relations among MZV's by $q$-shuffle relations mentioned above. In~\cite{C14}, the first named author of the present paper proved that all $\ok$-linear relations come from $k$-linear relations among MZV's of the same weight. Therefore, determination of the dimension of the $k$-vector space of MZV's of weight $w$ for $w\in \ZZ_{>0}$ is the central problem in this topic. Via numerical computation, Todd~\cite{To18} provided the following analogue of Zagier's dimension conjecture. 
 
\begin{conjecture}[Todd's dimension conjecture]\label{Todd's Conj} For any positive integer $w$, let $\cZ_{w}$ be the $k$-vector space spanned by all  MZV's of weight $w$. Define
\[
d_{w}' =
\begin{cases}
2^{w-1} & \textnormal {if $1\leq w < q$}, \\
2^{q-1}-1 & \textnormal{if $w=q$},\\
 \sum_{i=1}^{q} d'_{w-i} & \textnormal{if $w >q$}.
\end{cases}
\]
Then one has
\[ \dim_{k} \cZ_{w}=d_{w}'.\]
\end{conjecture}

In analogy with Hoffman's basis conjecture, Thakur~\cite{T17} proposed the following conjecture. 
\begin{conjecture}[Thakur's basis conjecture]\label{Thakur's Conj}
Let $w$ be a positive integer, and let $\ITw$ be the 
set consisting of all $\fs= (s_{1}, \ldots, s_{r} )\in \ZZ_{>0}^{r}$ (varying positive integers $r$) for which 
\begin{itemize}
    \item $\wt(\fs)=w$;
    \item $s_{i} \leq q$ for all $1 \leq  i \leq r - 1$; 
    \item $s_{r}<q$. 
    \end{itemize}
Then the following set
\[\mathcal{B}_{w}^{\rm{T}}:=\left\{\zeta_{A}(\fs)| \fs\in \ITw \right\} \]
is a basis of the $k$-vector space $\cZ_{w}$.
\end{conjecture}

We mention that for every positive integer $w$, the cardinality of $\IT_w$ is equal to $d_{w}'$ given in Conjecture~\ref{Todd's Conj}. It follows that as in the classical case, Thakur's basis conjecture implies Todd's dimension conjecture. The main result of this paper is to prove Thakur's basis conjecture.

\begin{theorem}\label{T:Main Thm}
For any positive integer $w$,  Conjecture~\ref{Thakur's Conj} is true. As a consequence, Conjecture~\ref{Todd's Conj} is also true. 
 \end{theorem}

As we have determined the dimension of $\cZ_{w}$ for each $w\in \NN$, it is a natural question about how to describe all the $k$-linear relations among the MZV's of  weight $w$. We establish a concrete and simple mechanism in Theorem~\ref{T: DetermineLR} that~\eqref{E:Main Relations} account for all the $k$-linear relations. Note that Theorem~\ref{T: DetermineLR} basically verifies the $\sB^{*}$-version of~\cite[Conjecture~5.1]{To18}.

\begin{remark}
In~\cite{ND21}, Ngo Dac showed that for each $w\geq 1$, $\mathcal{B}_{w}^{\rm{T}}$ is a generating set for the $k$-vector space $\cZ_{w}$. As a consequence of Ngo Dac's result, one has the {\it{upper bound result}}:
\[ \dim_{w}\cZ_{w}\leq d_{w}',\ \forall w\geq 1. \]
Theoretically, to prove Thakur's basis conjecture it suffices to show that his conjectural basis is linearly independent over $k$. However, when developing our approaches in a unified framework, we reprove Ngo Dac's result mentioned above in Corollary~\ref{Cor:generating set}, which also provides a generating set for the $k$-vector space spanned by the special values $\Li_{\fs}(\bone)$ for all indices $\fs$ with $\wt(\fs)=w$ and which is described in the next section.
\end{remark}

\begin{remark}

As mentioned above, by~\cite{C14} $k$-linear independence of MZV's implies $\ok$-linear independence. It follows that for each positive integer $w$, $\mathcal{B}_{w}^{\rm{T}}$ is also a basis for the $\ok$-vector space spanned by the MZV's of weight $w$.
\end{remark}

Given a finite place $v$ of $k$, we mention that $v$-adic MZV's $\zeta_{A}(\fs)_{v}$ were introduced in~\cite{CM21} for $\fs\in \ZZ_{>0}^{r}$. By~\cite[Thm.~1.2.2]{CM21}, there is a natural $k$-linear map  from ($\infty$-adic) MZV's to $v$-adic MZV's, and it was further shown that the map is indeed an algebra homomorphism with kernel containing the ideal generated by $\zeta_{A}(q-1)$ since $\zeta_{A}(q - 1)_{v} = 0$ by \cite{Go79}. As a consequence of Theorem~\ref{T:Main Thm}, we have the following upper bound result for $v$-adic MZV's.
\begin{corollary}
Let $v$ be a finite place of $k$. For a positive integer $w$, we let $\cZ_{v,w}$ be the $\ok$-vector space space by $v$-adic MZV's of weight $w$ defined in~\cite{CM21}. Then we have \[ \dim_{k} \cZ_{v,w}\leq d_{w}'-d_{w-(q-1)}', \]
where $d_{0}':=1$ and $d_{n}':=0$ for $n< 0$. 

\end{corollary}

\subsection{Strategy of proofs}

The key ingredient of our proof of Theorem~\ref{T:Main Thm} is to switch the study of MZV's to that of the special values of Carlitz multiple polylogarithms (abbreviated as CMPL's) $\Li_{\fs}$ for $\fs\in \ZZ_{>0}^{r}$ defined by the first named author in~\cite{C14}. For the definition of $\Li_{\fs}$, see~\eqref{E:CMPL}. Note that CMPL's are higher depth generalization of the Carlitz polylogarithms initiated by Anderson-Thakur~\cite{AT90}. Based on the interpolation formula of Anderson-Thakur~\cite{AT90, AT09}, one knows from~\cite{C14} that $\zeta_{A}(\fs)$ can be expressed as a $k$-linear combinations of CMPL's at some integral points, and in the particular case for $\fs \in \ITw$, one has the simple identity \eqref{E:sL=zeta} that
\[\zeta_{A}(\fs) =\Li_{\fs}(\bone), \]
where $\bone := (1,\ldots,1) \in \ZZ_{>0}^{\dep(\fs)}$.   

Given a positive integer $w$, we let $\INDw$ be the set consisting of all tuples of positive integers $\fs = (s_{1}, \ldots, s_{r})$ with $\wt(\fs)=w$ and $q \nmid s_{i}$ for all $1 \leq i \leq r$, and note that $\INDw$ was used and studied in~\cite{ND21}. The overall arguments in the proof of Theorem~\ref{T:Main Thm} are divided into the following two parts:

\begin{itemize}
    \item [(I)] We show  in Corollary~\ref{Cor:generating set} that the two sets $ \mathcal{B}_{w}^{\Li,\rm{T}}:=\left\{\Li_{\fs}| \fs\in \ITw \right\}$ and $\mathcal{B}_{w}^{T}$ are generating sets of the $k$-vector space $\cZ_{w}$ for every positive integer $w$ (the latter one was known by Ngo Dac~\cite{ND21}), and further show in Theorem~\ref{theorem-generator} that the set $\left\{ \Li_{\fs}(\bone)| \fs\in \INDw  \right\}$ is a generating set for $\cZ_{w}$.
    
    \item [(II)] We prove in Theorem~\ref{theorem-IND-basis} that  $\left\{ \Li_{\fs}(\bone)| \fs\in \INDw  \right\}$ is a linearly independent set over $k$. Since $|\ITw|=|\INDw|$ (see~Proposition~\ref{Pop:|ITw|=|INDw|}), it follows that $\mathcal{B}_{w}^{T}$ is a $k$-basis of $\cZ_{w}$.  
\end{itemize}

Regarding the part (I) above, we try to abstract and unify our methods to have a wide scope. We consider a formal $k$-space $\cH$ generated by indices, on which the harmonic product $*^{\Li}$ and $q$-shuffle product $*^{\zeta}$ are defined. We also consider the power/truncated sums maps arising from $\Li_{\fs}(\bone)$ and $\zeta(\fs)$. We then define the maps $\sLL, \sLZ \colon \cH\rightarrow k_{\infty}$ as the CMPL and MZV-realizations. These  realizations preserve the harmonic product and $q$-shuffle product respectively, illustrating the stuffle relations~\cite{C14} for CMPL's and the $q$-shuffle relations~\cite{T10, Ch15} for MZV's in a formal framework.  See Theorem~\ref{P:product of sLbullet}.

Then we adopt the ideas and methods rooted in~\cite{To18, ND21} to achieve Theorem~\ref{theorem-algo}, which enables us to show the results mentioned in (I). Since some essential arguments are from those in~\cite{ND21}, we leave the detailed proofs in the appendix. However, under such abstraction the maps $\sUB$ given in Definition~\ref{def-sU} are explicit. In particular, Theorem~\ref{theorem-algo}~(2) allows us to provide a concrete and effective way to express $\Li_{\fs}(\bone)$ (resp.~$\zeta_{A}(\fs)$) as  linear combinations in terms of $\mathcal{B}_{w}^{\Li,\rm{T}}$ (resp.~$\mathcal{B}_{w}^{T}$) simultaneously. This is more concrete and explicit than those developed in~\cite{ND21}, and enables us to establish a simple mechanism in Theorem~\ref{T: DetermineLR} that generates all the $k$-linear relations among MZV's of the same weight (cf.~the $\sB^{*}$-version of~\cite[Conjecture~5.1]{To18}).

Concerning the second part (II), the primary tool that we use is the Anderson-Brownawell-Papanikolas criterion~\cite[Thm.~3.1.1]{ABP04} (abbreviated as ABP-criterion), which has very strong applications in transcendence theory in positive characteristic. We prove (II) by induction on $w$. We start with a linear equation 
 \begin{equation}\label{E:IntrodtionLE}
 \sum_{\fs\in \INDw }\alpha_{\fs}(\theta) \Li_{\fs}(\bone)=0,\ \hbox{ for } \alpha_{\fs}=\alpha_{\fs}(t)\in \FF_{q}(t), \ \forall \fs\in \INDw, 
 \end{equation} 
and aim to show that $\alpha_{\fs}=0$ for all $\fs\in \INDw$. We outline our arguments as below.

\begin{enumerate}
    \item[(II-1)] Using the period interpretation~\cite[(3.4.5)]{C14} for special values of CMPL's (inspired by~\cite{AT09} for MZV's), we can construct a system of difference equations $\widetilde{\psi}^{(-1)}=\widetilde{\Phi} \widetilde{\psi}$ fitting into the conditions of ABP-criterion for which~\eqref{E:IntrodtionLE} can be expressed as a $k$-linear identity among the entries of $\widetilde{\psi}(\theta)$. We then apply the arguments  in the proof of~\cite[Thm.~2.5.2]{CPY19}, eventually we obtain the big system of Frobenius difference equations~\eqref{eq-Frob-w}, which is essentially from~\cite[(3.1.3)]{CPY19}.
\item [(II-2)] Denote by $\sX_{w}$ the $\FF_{q}(t)$-vector space of the solution space of~\eqref{eq-Frob-w}. We then show in Theorem~\ref{theorem-dimension} that $\dim_{\FF_{q}(t)}\sX_{w}$ is either $1$ if $(q-1) \mid w$, or $0$ if $(q-1) \nmid w$. 
    \item [(II-3)] Using the trick of~\cite{C14, CPY19} together with the induction hypothesis, we establish Lemma~\ref{lemma-rational}. By the first part of Lemma~\ref{lemma-rational}, we see that if $\alpha_{\fs} \neq 0$, then  $\fs$ must be in $\INDzw$, which is given in~\eqref{E:INDzw}. Then we apply Lemma~\ref{lemma-rational}~(2) to conclude that there exists $(\varepsilon_{\fs})\in \sX_w$ for which     $\alpha_{\fs}=\varepsilon_{\fs}$ for all $\fs \in \INDzw$.
     Note that the description of indices in $\INDzw$ arises from the simultaneously Eulerian phenomenon in~\cite[Cor.~4.2.3]{CPY19}, which was first witnessed by Lara Rodr\'iguez and Thakur in~\cite{LRT14}.
    \item [(II-4)] When $(q - 1) \nmid w$, we use the fact of this case that $\sX_{w}=0$ to conclude that $\varepsilon_{\fs}=0$ for every $\fs$, and hence $\alpha_{\fs}(\theta)=0$ for every $\fs$. When $(q-1) \mid w$, we apply Theorem~\ref{theorem-generator} together with the fact of this case that $\dim_{\FF_{q}(t)}\sX_{w}=1$. Having one trick further in the proof of Theorem~\ref{theorem-IND-basis} shows that $\alpha_{\fs}(\theta)=0$ for every $\fs$.
\end{enumerate}
\begin{remark}
The above is the overall strategy of our proof. However, to save some length without affecting logical arguments, we avoid many details in II-(1). Instead, we directly go to~\eqref{E:IntrodtionLE} in this paper.
\end{remark}

\subsection{Organization of the paper}
    In Sec.~\ref{Sec: Indices}, we introduce the abstract $k$-vector space $\cH$, on which  the harmonic product $*^{\Li}$ and $q$-shuffle product $*^{\zeta}$ are defined. The purpose of Sec.~\ref{Sec: Indices} is  to derive the product formulae, stated as Proposition~\ref{P:product of sLbullet}, for the CMPL and MZV realizations $\sLL$ and $\sLZ$.  In Sec.~\ref{Sec: Generators} we follow \cite{To18, ND21} to transfer their $\sB, \sC, \sBC$ maps into our formal framework in a concrete way. The primary result of this section is to establish Theorem~\ref{theorem-generator}, which is an application of Theorem~\ref{theorem-algo}, whose detailed proof is given in the appendix. 
    
   We study the specific system of Frobenius difference equations~\eqref{eq-Frob-w} mentioned in the (II) above in Sec.~\ref{Sec:FDE}. We carefully analyze the solution space $\sX_{w}$ of  \eqref{eq-Frob-w}, and determine its dimension in Theorem~\ref{theorem-dimension}.   In Sec.~\ref{Sec:Linear Indep}, we establish the key Lemma~\ref{lemma-rational}, which is used to prove the $k$-linear independence of $\left\{ \Li_{\fs}(\bone)| \fs\in \INDw  \right\}$ in Theorem~\ref{theorem-IND-basis}. With these results at hand, we prove Theorem~\ref{T:Main Thm} in Sec.~\ref{Sec: Proof of Main Thm} as a short conclusion. Finally, combining Theorems~\ref{theorem-algo} and \ref{T:Main Thm} we give a proof of Theorem~\ref{T: DetermineLR} in Sec.~\ref{Sub:Generating set of linear relations}. As mentioned above, the appendix consists of a detailed proof of Theorem~\ref{theorem-algo}.
   
 \begin{remark}
 When this paper was nearly in its final version, we announced our results to Thakur and soon after Ngo Dac sent his paper \cite{IKLNDP22} to Thakur on the same day announcing that he and his coauthors show the same results for alternating MZV's. Our paper was finished a few days later than  theirs. The key strategy of their proofs is in the same direction as ours. They switch the study of alternating MZV's to the $k$-vector space spanned by the following special values of CMPL's:
 \[\left\{ \Li_{\fs} (\gamma_{1},\ldots,\gamma_{r} )| \fs\in \ITw,\  (\gamma_{1},\ldots,\gamma_{r} )\in (\FF_{q^{q-1}}^{\times})^{r}\right\} .\] Note that any value in the set above is equal to an alternating MZV at $\fs$ up to an algebraic multiple (cf.~\cite[Prop.~2.12]{CH21}). However, by the first author's result~\cite[Thm.~5.4.3]{C14} showing that CMPL's at algebraic points form a $\ok$-graded algebra defined over $k$, one can remove the algebraic factor without affecting the study for the $k$-vector space of alternating MZV's. We find that some directions of  our ideas and theirs are similar and the primary methods rooted in~\cite{ABP04, C14, To18, CPY19, ND21} are the same, but presentations and detailed arguments are  different.  
 
 \end{remark}

\section{Two products on $\cH$}

\subsection{Indices}\label{Sec: Indices}

By an index, we mean the empty set $\emptyset$ or an $r$-tuple of positive integers $\fs=(s_{1},\ldots,s_{r})$. In the former case, its depth and weight are defined to be $\dep(\emptyset)=0$ and $\wt(\emptyset)=0$. The depth and weight of the latter case are defined to be 
$\dep(\fs)=r$ and $\wt(\fs)=\sum_{i=1}^{r}s_{i}$ respectively. We denote by $\cI := \bigsqcup_{r \geq 0} \ZZ_{\geq 1}^{r}$ the set of indices, where $\ZZ_{\geq 0}^{0}$ is referred to the empty set $\emptyset$.

Throughout this paper, we adapt the following notations. 

\begin{align*}
\cI_{w} &:= \{ \fn \in \cI \ | \ \wt(\fn) = w \} \ \ (w \geq 0), \\
\cI_{> 0} &:= \{ \fn \in \cI \ | \ \wt(\fn) > 0 \} = \cI \setminus \{ \emptyset \}, \\
\ITw &:= \{ (s_{1}, \ldots, s_{r} ) \in \cI_{w} \ | \ s_{i} \leq q \ (1 \leq \forall i \leq r - 1) \ \textrm{and} \ s_{r} < q \} \ \ (w > 0), \\
\INDw &:= \{ (s_{1}, \ldots, s_{r} ) \in \cI_{w} \ | \ q \nmid s_{i} \ (1 \leq \forall i \leq r) \} \ \ (w > 0), \\
\IT_{0} &:= \IND_{0} := \{ \emptyset \}.
\end{align*}
We mention that $\ITw$ refers to the set of indexes arising from Thakur's basis, and $\INDw$ is the set studied in~\cite{ND21}. For any subset $S$ of $\cI$, we denote by $|S|$ the cardinality of $S$ when no confusions arise.

\begin{proposition}\label{Pop:|ITw|=|INDw|}
For each weight $w\geq 0$, we have a bijection between $\INDw$ and $\ITw$, and hence $|\INDw|=|\ITw|$.
\end{proposition}

\begin{proof}
The desired result follows from the following correspondence
\begin{align*}
(m_{1} q + n_{1}, \ldots, m_{r} q + n_{r}) \longleftrightarrow (q^{\{ m_{1} \}}, n_{1}, \ldots, q^{\{ m_{r} \}}, n_{r}) \ \ \ (m_{i} \geq 0, \ 1 \leq n_{i} \leq q - 1),
\end{align*}
where $q^{\{m\}}$ denotes the sequence $(q, \ldots, q) \in \ZZ_{\geq 1}^{m}$. It is understood that in the case of $m=0$, $q^{\left\{0 \right\}}$ is referred to the empty index $\emptyset$.
\end{proof}

Given $\ell$ indices $\fs_{1},\ldots,\fs_{\ell} \in \cI$ with  $\fs_{i} = (s_{i1}, \ldots, s_{ir_{i}})$ ($1 \leq i \leq \ell$), we define
\begin{align}\label{E:s1,..,sl}
(\fs_{1}, \ldots \fs_{\ell}) := (s_{11}, \ldots, s_{1r_{1}}, s_{21}, \ldots, s_{2r_{2}}, \ldots, s_{\ell 1}, \ldots, s_{\ell r_{\ell}})
\end{align} to be the index obtained by putting the given indices consecutively.
For each non-empty index $\fs = (s_{1}, \ldots, s_{r}) \in \cI_{> 0}$, we define
\begin{equation}
\fs_{+} := (s_{1}, \ldots, s_{r - 1}) \ \textrm{and} \ \fs_{-} := (s_{2}, \ldots, s_{r}).
\end{equation}
In the depth one case, we note that $\fs_{+}=\emptyset$ and $\fs_{-}=\emptyset$.

Let $\cH = \bigoplus_{w \geq 0} \cH_{w}$ be the $k$-vector space with basis $\cI$ graded by weight and let $\cH_{> 0} := \bigoplus_{w > 0} \cH_{w}$.
For each index $\fs = (s_{1}, \ldots, s_{r}) \in \cI$, the corresponding generator in $\cH$ is denoted by $[\fs] = [s_{1}, \ldots, s_{r}]$ or $\fs$.
For each $P = \sum_{\fs \in \cI} a_{\fs} [\fs] \in \cH$, the support of $P$ is defined by
\begin{align*}
\Supp(P) := \{ \fs \in \cI \ | \ a_{\fs} \neq 0 \}.
\end{align*}

\begin{definition}\label{Def:multi-linear[,]}
Given a positive integer $\ell$, we define the following $k$-multilinear map

\[
[-,-, \ldots,-]:\cH^{\oplus \ell}\rightarrow \cH  \]
as follows. For any indices  $\fs_{1},\ldots,\fs_{\ell}\in \cI$, let $(\fs_{1},\ldots,\fs_{\ell})$ be given in~\eqref{E:s1,..,sl}. We define
\[ [\fs_{1},\ldots,\fs_{\ell}]:= [(\fs_{1},\ldots,\fs_{\ell})]\in \cH.
\] 
\end{definition}

\begin{remark}\label{Rem: [s,0]=0}
As the map above is multilinear, for any $\fs\in\cI$ we particularly have the following identity 
\[ [\fs,0]=0\in \cH \]
which will be used in the appendix.  As $\emptyset$ is also an index by our definition, we mention that 

\[[\fs,\emptyset] \neq [\fs,0]=0\in \cH. \]
\end{remark}

\subsection{The maps $\sLL$ and $\sLZ$}
Recall that the Carlitz logarithm is given by
\[\log_{C}(z):=\sum_{d\geq 0} \frac{z^{q^{d}}}{L_{d}} \in \power{k}{z} ,\]
where $L_{0} := 1$ and $L_{d} := (\theta - \theta^{q}) \cdots (\theta - \theta^{q^{d}})$ for $d \geq 1$ (see~\cite{Go96, T04}), and the $n$th Carlitz polylogarithm defined by Anderson-Thakur~\cite{AT90} is the following power series
\[ \Li_{n}(z):= \sum_{d \geq 0}\frac{z^{q^{d}}}{L_{d}^{n}}\in \power{k}{z}. \]
For any index $\fs=(s_{1},\ldots,s_{r})\in \cI_{>0}$ of positive depth, the $\fs$th Carlitz multiple polylogarithm (abbreviated as CMPL) is given by the following series (see~\cite{C14})

\begin{equation}\label{E:CMPL}
 \Li_{\fs}(z_{1},\ldots,z_{r}):=\sum_{d_{1} > \cdots > d_{r} \geq 0} \dfrac{z_{1}^{q^{d_1}} \cdots z_{r}^{q^{d_r}}}{L_{d_{1}}^{s_{1}} \cdots L_{d_{r}}^{s_{r}}}\in \power{k}{z_{1},\ldots,z_{r}} . 
\end{equation}

For each $s \in \ZZ_{\geq 1}$ and $d \in \ZZ_{\geq 0}$, we set
\begin{align*}
S^{\Li}_{d}(s) := \dfrac{1}{L_{d}^{s}}\in k
\ \ \textrm{and} \ \
S^{\zeta}_{d}(s) := \sum_{a \in A_{+, d}} \dfrac{1}{a^{d}}\in k,
\end{align*}
where $A_{+, d}$ is the set of monic polynomials in $A$ of degree $d$.
We then define $k$-linear maps $\sLB_{d}$, $\sLB_{< d}$ and $\sLB$ on $\cH$.

\begin{definition}\label{Def:sLB}
Let $\bullet \in \{ \Li, \zeta \}$ and $d \in \ZZ$.
\begin{enumerate}
\item
The map $\sLB_{d} \colon \cH \to k$ is the $k$-linear map defined by
\begin{align}\label{E:sLBd}
\sLB_{d}(\fs) &:= \left\{ \begin{array}{cl} {\displaystyle \sum_{d = d_{1} > \cdots > d_{r} \geq 0}} \SB_{d_{1}}(s_{1}) \cdots \SB_{d_{r}}(s_{r}) & (\fs = (s_{1}, \ldots, s_{r}) \in \cI_{> 0} \ \textrm{and} \ d \geq \dep(\fs) - 1) \\ 1 & (\fs = \emptyset \ \textrm{and} \ d = 0) \\ 0 & (\textrm{otherwise}) \end{array} \right.
\end{align}

\item
The map $\sLB_{< d} \colon \cH \to k$ is the $k$-linear map defined by
\begin{align}\label{E:sLB<d}
\sLB_{< d}(\fs) &:= \left\{ \begin{array}{cl} {\displaystyle \sum_{d > d_{1} > \cdots > d_{r} \geq 0}} \SB_{d_{1}}(s_{1}) \cdots \SB_{d_{r}}(s_{r}) & (\fs = (s_{1}, \ldots, s_{r}) \in \cI_{> 0} \ \textrm{and} \ d \geq \dep(\fs)) \\ 1 & (\fs = \emptyset \ \textrm{and} \ d \geq 1) \\ 0 & (\textrm{otherwise}) \end{array} \right.
\end{align}

\item
The map $\sLB \colon \cH \to k_{\infty}$ is the $k$-linear map defined by
\begin{align}\label{E:sLB}
\sLB(\fs) &:= \left\{ \begin{array}{cl} {\displaystyle \sum_{d_{1} > \cdots > d_{r} \geq 0}} \SB_{d_{1}}(s_{1}) \cdots \SB_{d_{r}}(s_{r}) & (\fs = (s_{1}, \ldots, s_{r}) \in \cI_{> 0}) \\ 1 & (\fs = \emptyset) \\ 0 & (\textrm{otherwise}). \end{array} \right.
\end{align}
\end{enumerate}
\end{definition}
It follows that for each $P \in \cH$, $\fs \in \cI$, $s \in \ZZ_{\geq 1}$ and $d \in \ZZ$, we have
\begin{align*}
\sLB_{d}(P) &= \sLB_{< d + 1}(P) - \sLB_{< d}(P), \\
\sLB_{< d}(P) &= \sum_{d' < d} \sLB_{d'}(P) = \sum_{0 \leq d' < d} \sLB_{d'}(P), \\
\sLB(P) &= \sum_{d \in \ZZ} \sLB_{d}(P) = \lim_{d \to \infty} \sLB_{< d}(P) \in k_{\infty}, \\
\sLB_{d}(s) \sLB_{< d}(P) &= \sLB_{d}([s, P]), \\
\sLL(\fs) &= \Li_{\fs}(\bone), \\
\sLZ(\fs) &= \zeta_{A}(\fs),
\end{align*}
where $\bf{1}$ is simply referred to $(1,\ldots,1)\in \ZZ_{>0}^{\dep \fs}$ when it is clear from the context without confusion, and we set $\Li_{\emptyset}(\bone) := \zeta_{A}(\emptyset) := 1$.

\begin{remark} \label{rmk-C=A}
For $1 \leq s \leq q$ and $d \geq 0$, 
we have the following equality due to Carlitz (see also \cite[Thm.~5.9.1]{T04})
\begin{align*}
\dfrac{1}{L_{d}^{s}} = \sum_{a \in A_{+, d}} \dfrac{1}{a^{s}}.
\end{align*}
Therefore, for each $\fs = (s_{1}, \ldots, s_{r})$ with $1 \leq s_{i} \leq q$ and $d \in \ZZ$, we have
\begin{align*}
\sLL_{d}(\fs) = \sLZ_{d}(\fs), \ \ \ \sLL_{< d}(\fs) = \sLZ_{< d}(\fs), 
\end{align*}
and
\begin{equation}\label{E:sL=zeta}
\Li_{\fs}(\bone)=\sLL(\fs) = \sLZ(\fs)=\zeta_{A}(\fs). 
\end{equation}

\end{remark}

\subsection{Product formulae}

The harmonic product on $\cH$ is denoted by $*$ or $\sLi$.
Thus $*$ is a $k$-bilinear map $\cH \times \cH \to \cH$ such that
\begin{align*}
&[\emptyset] * P = P * [\emptyset] = P, \\
&[\fs] * [\fn] = [s_{1}, \fs_{-} * \fn] + [n_{1}, \fs * \fn_{-}] + [s_{1} + n_{1}, \fs_{-} * \fn_{-}]
\end{align*}
for each $P \in \cH$, $\fs = (s_{1}, \fs_{-}), \fn = (n_{1}, \fn_{-}) \in I_{> 0}$.
For each $\fs, \fn \in I_{> 0}$, we put
\begin{equation}\label{E:DLi}
D^{\Li}_{\fs, \fn} := 0 \in \cH.
\end{equation}

For each $s, n \geq 1$, H.-J. Chen showed in \cite{Ch15} that
\begin{equation}\label{E:Chen}
\sLZ_{d}(s) \sLZ_{d}(n) = \sLZ_{d}(s + n) + \sum_{j = 1}^{s + n - 1} \Delta_{s, n}^{[j]} \sLZ_{d}(s + n - j, j),
\end{equation}
where we set
\begin{align*}
\Delta_{s, n}^{[j]} := \left\{ \begin{array}{ll} {\displaystyle (- 1)^{s - 1} \binom{j - 1}{s - 1} + (- 1)^{n - 1} \binom{j - 1}{n - 1}} & \textrm{if $(q - 1) \mid k$ and $1 \leq j < s + n$} \\ 0 & \textrm{otherwise} \end{array} \right..
\end{align*}
The $q$-shuffle product on $\cH$ is denoted by $\sz$.
Thus $\sz$ is a $k$-bilinear map $\cH \times \cH \to \cH$ such that
\begin{align*}
&[\emptyset] *^{\zeta} P = P *^{\zeta} [\emptyset] = P, \\
&[\fs] *^{\zeta} [\fn] = [s_{1}, \fs_{-} *^{\zeta} \fn] + [n_{1}, \fs *^{\zeta} \fn_{-}] + [s_{1} + n_{1}, \fs_{-} *^{\zeta} \fn_{-}] + D^{\zeta}_{\fs, \fn}
\end{align*}
for each $P \in \cH$, $ \fs = (s_{1}, \fs_{-}),\fn = (n_{1}, \fn_{-}) \in \cI_{> 0}$, where we set
\begin{equation}\label{E:DZetaS}
D^{\zeta}_{\fs, \fn} := \sum_{j = 1}^{s_{1} + n_{1} - 1} \Delta_{s_{1}, n_{1}}^{[j]} [s_{1} + n_{1} - j, (j) \sz (\fs_{-} *^{\zeta} \fn_{-})].
\end{equation}
By the induction on $w + w'$, we can show that $P *^{\bullet} Q \in \cH_{w + w'}$ for each $w, w' \in \ZZ_{\geq 0}$, $P \in \cH_{w}$, $Q \in \cH_{w'}$ and $\bullet \in \{ \Li, \zeta \}$. In particular, when $w, w' \in \ZZ_{\geq 1}$ we have $D^{\zeta}_{\fs, \fn} \in \cH_{w + w'}$ for each $\fs \in \cI_{w}$ and $\fn \in \cI_{w'}$.

\begin{remark} \label{remark-D-vanish}
When $s + n \leq q$, we observe that $\Delta_{s, n}^{[j]} = 0$ for all $j$.
Thus for $\fs = (s_{1}, \ldots), \fn = (n_{1}, \ldots) \in \cI_{> 0}$, if $s_{1} + n_{1} \leq q$ then we have $D^{\zeta}_{\fs, \fn} = 0$.
\end{remark}

The following product formulae are crucial when proving Theorem~\ref{theorem-algo}, whose detailed proof is given in the appendix. 
\begin{proposition}\label{P:product of sLbullet} Let $\bullet\in \left\{\Li, \zeta \right\}$. Given any $P,Q\in \cH$ and $\fs = (s_{1}, \fs_{-}), \fn = (n_{1}, \fn_{-}) \in \cI_{> 0}$, we have the following identities for each $d\in \ZZ$.
\begin{enumerate}
\item $\sLB(P) \sLB(Q) = \sLB(P *^{\bullet} Q)$.

\item $
\sLB_{< d}(P) \sLB_{< d}(Q) = \sLB_{< d}(P *^{\bullet} Q)$. 
\item
\begin{align*}
\sLB_{d}(\fs) \sLB_{d}(\fn) &= \sLB_{d}([s_{1} + n_{1}, \fs_{-} *^{\bullet} \fn_{-}]) + \sLB_{d}(D^{\bullet}_{\fs, \fn}) \\
&= \sLB_{d}(\fs *^{\bullet} \fn) - \sLB_{d}([s_{1}, \fs_{-} *^{\bullet} \fn]) - \sLB_{d}([n_{1}, \fs *^{\bullet} \fn_{-}]).
\end{align*}
\end{enumerate}

\end{proposition}

\begin{proof}
We first mention that (3) for $d$ follows from the second one for the same $d$.
Indeed, for each $\fs = (s_{1}, \ldots), \fn = (n_{1}, \ldots) \in \cI_{> 0}$ we have
\begin{align*}
\sLL_{d}(\fs) \sLL_{d}(\fn)
&= \sLL_{d}(s_{1}) \sLL_{d}(n_{1}) \sLL_{< d}(\fs_{-}) \sLL_{< d}(\fn_{-}) \\
&= \sLL_{d}(s_{1} + n_{1}) \sLL_{< d}(\fs_{-} *^{\Li} \fn_{-}) \\
&= \sLL_{d}([s_{1} + n_{1}, \fs_{-} *^{\Li} \fn_{-}]) + \sLL_{d}(D^{\Li}_{\fs, \fn}),
 \end{align*}
where the second equality comes from (2) and the definition of $\sLL_{d}$, and the third identity comes from \eqref{E:DLi}. Similarly we have
\begin{align*}
\sLZ_{d}(\fs) \sLZ_{d}(\fn)
&= \sLZ_{d}(s_{1}) \sLZ_{d}(n_{1}) \sLZ_{< d}(\fs_{-}) \sLZ_{< d}(\fn_{-}) \\
&= \left( \sLZ_{d}(s_{1} + n_{1}) + \sum_{j} \Delta_{s_{1}, n_{1}}^{j} \sLZ_{d}(s_{1} + n_{1} - j) \sLZ_{< d}(j) \right) \sLZ_{< d}(\fs_{-} *^{\zeta} \fn_{-}) \\
&= \sLZ_{d}([s_{1} + n_{1}, \fs_{-} *^{\zeta} \fn_{-}]) + \sum_{j} \Delta_{s_{1}, n_{1}}^{[j]} \sLZ_{d}(s_{1} + n_{1} - j) \sLZ_{< d}((j) *^{\zeta} (\fs_{-} *^{\zeta} \fn_{-}) ) \\
&= \sLZ_{d}([s_{1} + n_{1}, \fs_{-} *^{\zeta} \fn_{-}]) + \sLZ(D^{\zeta}_{\fs, \fn}),
 \end{align*}
where the second equality comes from~\eqref{E:Chen} and (2), the third identity comes from (2) and~\eqref{E:DZetaS}.

To prove the formula~(2), we first mention that when $d \leq 0$, the formula holds as both sides of the identity are zero. We then prove the  formula~(2) by induction on $d$. Now, let $d$ be any positive integer.
By bi-linearity, we may assume that $P = \fs, Q = \fn \in \cI_{> 0}$.
Then we have
\begin{align*}
&\sLB_{< d}(\fs) \sLB_{< d}(\fn) \\
&= \sum_{d_{1} < d} \sLB_{d_{1}}(s_{1}) \sLB_{< d_{1}}(\fs_{-}) \sLB_{< d_{1}}(\fn)
+ \sum_{d_{1} < d} \sLB_{d_{1}}(n_{1}) \sLB_{< d_{1}}(\fs) \sLB_{< d_{1}}(\fn_{-})
+ \sum_{d_{1} < d} \sLB_{d_{1}}(\fs) \sLB_{d_{1}}(\fn) \\
&= \sum_{d_{1} < d} \sLB_{d_{1}}(s_{1}) \sLB_{< d_{1}}(\fs_{-} *^{\bullet} \fn)
+ \sum_{d_{1} < d} \sLB_{d_{1}}(n_{1}) \sLB_{< d_{1}}(\fs *^{\bullet} \fn_{-}) \\
& \ \ \ + \sum_{d_{1} < d} \sLB_{d_{1}}([s_{1} + n_{1}, \fs_{-} *^{\bullet} \fn_{-}] + D^{\bullet}_{\fs, \fn}) \\
&= \sLB_{< d}(\fs *^{\bullet} \fn),
\end{align*}
where the second equality comes from the induction hypothesis as well as (3) (as (2) holds for $d_1<d$ by induction hypothesis).
Finally, the formula~(1) follows from (2) by taking the limit $d \to \infty$.
\end{proof}

\section{Generators}\label{Sec: Generators}
The purpose of this section is to establish Theorem~\ref{theorem-algo}, which allows us to obtain the desired generating sets for $\cZ_w$. To achieve it, we need to set up the box-plus operator on $\cH^{\oplus 2}$ as well as the $k$-linear maps $\sUB$ on $\cH$.

\subsection{The box-plus operator}
\begin{definition}\label{Def:boxplus}

Let $\boxplus \colon \cH^{\oplus 2} \to \cH$ be the $k$-linear map defined by
\begin{align*}
\emptyset \boxplus P = P \boxplus \emptyset := 0
\ \ \textrm{and} \ \
\fs \boxplus \fn := (\fs_{+}, s_{r} + n_{1}, \fn_{-})
\end{align*}
for $P \in \cH$ and $\fs = (\fs_{+}, s_{r}), \fn = (n_{1}, \fn_{-}) \in \cI_{> 0}$. 
\end{definition}

\begin{remark}
To avoid confusion, we mention that for each $n \geq 1$ and $P = \sum_{\fs \in \cI} a_{\fs} [\fs] \in \cH$, 
\begin{align*}
[n, P] & = \sum_{\fs = (s_{1}, \ldots, s_{r}) \in \cI} a_{\fs} [n, s_{1}, \ldots, s_{r}] \in \cH_{> 0}, \\
(n) \boxplus P & = \sum_{\fs = (s_{1}, \ldots, s_{r}) \in \cI_{> 0}} a_{\fs} [n + s_{1}, s_{2}, \ldots, s_{r}] \in \cH_{> 0} \ \ \ (\neq [n] + P).
\end{align*}
Furthermore,  in the depth one case for $\fs=(s)$ and $\fn=(n)$, we have $\fs_{+}=\emptyset=\fn_{-}$ and hence in this case \[\fs \boxplus \fn=[s+n].\]
\end{remark}

We define a $k$-linear endomorphism $\sU^{\bullet} \colon \cH \to \cH$ as follows.
Let $\fs \in \cI$. 
We write 
\begin{equation}\label{E:s^T}
\fs = (s_{1}, \ldots ) = (\fs^{\rT}, q^{\{ m \}}, \fs')
\end{equation}
with $\fs^{\rT} \in \IT$, $m \geq 0$, and $\fs' = (s'_{1}, \ldots)$ with $s'_{1} > q$ or $\fs' = \emptyset$. 
When $\fs' \neq \emptyset$, we set
\begin{equation}\label{E:s''}
\fs''  := (s'_{1} - q, \fs'_{-}).
\end{equation}

\begin{example}
    Let $\fs=(q-1,q,q,q+1,1)$. Then $\fs^{\rT}=(q-1)$, $m=2$, and $\fs'=(q+1,1)$. Since $\fs'\neq\emptyset$, we also have $\fs''=(1,1)$.
\end{example}

\subsection{Generating sets}

We set 
\begin{equation}\label{E:alpha}
\alpha^{\bullet}_{q}(P) := [1, (q - 1) *^{\bullet} P]
\end{equation}
for each $P \in \cH$. For $m=0$, $\alpha_{q}^{\bullet, 0}$ is defined to be the identity map on $\cH$, and for $m\in \ZZ_{>0}$, $\alpha^{\bullet,m}_{q}$  is defined to be the $m$th iteration of $\alpha^{\bullet}_{q}$.

\begin{definition} \label{def-sU}
For $\bullet\in \left\{  \Li, \zeta \right\}$, we define the $k$-linear map $\sUB:\cH\rightarrow \cH$ given by
\begin{align*}
\sUB(\fs) := \left\{ \begin{array}{@{}ll}
- [\fs^{\rT}, q^{\{m + 1\}}, \fs''] + L_{1}^{m + 1} [\fs^{\rT}, \alpha_{q}^{\bullet, m + 1}(\fs'')] & \\
\hspace{5.0em} + L_{1}^{m + 1} (\fs^{\rT} \boxplus \alpha_{q}^{\bullet, m + 1}(\fs'')) - [\fs^{\rT}, q^{\{m\}}, D^{\bullet}_{q, \fs''}] & (\fs' \neq \emptyset) \\[0.5em]
L_{1}^{m} [\fs^{\rT}, \alpha_{q}^{\bullet, m}(\emptyset)] + L_{1}^{m} (\fs^{\rT} \boxplus \alpha_{q}^{\bullet, m}(\emptyset)) & (\fs' = \emptyset)
\end{array} \right.
\end{align*}
\end{definition}
\begin{remark}\label{Rem:Us=s}
From the definition, one sees that $\sU^{\bullet}(\cH_{w}) \subset \cH_{w}$. We also mention that when $\fs\in \IT$, then $\sUB(\fs)=\fs$ since in this case, we have $m=0$ and $\fs'=\emptyset$.

\end{remark}

\begin{theorem} \label{theorem-algo}
Let  $\bullet\in \left\{ \Li, \zeta \right\}$. For each $P \in \cH$, the following statements hold.
\begin{enumerate}
\item $\sLB(\sUB(P)) = \sLB(P)$.
\item There exists an explicit integer $e \geq 0$ such that $\Supp((\sU^{\bullet})^{e}(P)) \subset \IT$, where $(\sU^{\bullet})^{0}$ is defined to be the identity map and for $e\in\ZZ_{>0}$, $(\sU^{\bullet})^{e}$ is defined to be the $e$th iteration of $\sU^{\bullet}$.
\end{enumerate}
\end{theorem}

Note that our $\sUB$ is concrete and explicit, and simultaneously deals with the two products $*^{\Li}$ and $*^{\zeta}$.
Our strategy  for proving Theorem \ref{theorem-algo} arises from~\cite{ND21}, and so we leave the detailed proof to the appendix. However, we give an outline of the proof of Theorem \ref{theorem-algo} (1) as follows. Let $\bullet\in \left\{\Li, \zeta \right\}$.
\begin{enumerate}
\item[(I)] We first define the space of {\it{binary relations}} $\cP^{\bullet}\subset \cH^{\oplus 2}$ in~\eqref{E:Pbullet}, and note that for $(P,Q)\in \cP^{\bullet}$, we have
\[\sLB(P, Q) := \sLB(P)+\sLB(Q)=0  .\] 

\item[(II)] For any $\fs\in \cI_{>0}$ and integer $m\geq 0$, we define in Sec.~\ref{Sec:Maps B,C,BC} the following maps
\[ \sB_{\fs}^{\bullet},\sC_{\fs}^{\bullet},\sBC_{q}^{\bullet,m}:\cH_{>0}^{\oplus 2}\rightarrow \cH_{>0}^{\oplus 2}, \]
and further show in Proposition~\ref{prop-BC} that they preserve $\cP^{\bullet}$. 
\item[(III)] In order to give a clear outline, we drop the subscripts to avoid heavy notation as all details are given in the proof of Theorem~\ref{theorem-sLsU}. We start with $R_{1}\in \cP_{q}^{\bullet}$, and consider $\sDB(R_{1})$ for choosing a suitable $\sDB$ arising from one of $\sB^{\bullet}, \sC^{\bullet}, \sBC^{\bullet}$ as well as their composites (see~\eqref{E:Appendix EQ1}, \eqref{E:Appendix EQ2}, \eqref{E:Appendix EQ3} and \eqref{E:Appendix EQ4}).  Based on (II), we have $\sDB(R_{1})\in \cP^{\bullet}$ and hence $\sLB(\sDB(R_{1}))=0$. In this case, the expansion of $\sLB(\sDB(R_{1}))$ is equal to $\sLB(\fs)-\sLB(\sUB(\fs))$, and hence $\sLB(\fs)=\sLB(\sUB(\fs))$ as desired.
\end{enumerate}

As a consequence of Theorem~\ref{theorem-algo}, we obtain the following important equality. 
\begin{corollary}\label{Cor:generating set} Let $\bullet\in \left\{\Li, \zeta\right\}$. For any $w\geq 0$, we let  $\cZ_{w}^{\bullet} $ be the $k$-vector space spanned by the elements $\sLB(\fs)\in k_{\infty}$  for $\fs \in \cI_{w}$. Then 
\[\left\{ \sLB(\fs)| \fs\in \ITw \right\}\] is a generating set for  $\cZ_{w}^{\bullet} $. In particular, we have \[\cZ^{\Li}_{w}=\cZ_{w}^{\zeta} \ (=\cZ_{w}) .\]
\end{corollary}
\begin{proof}
The first assertion comes from Theorem~\ref{theorem-algo}. We note that for $\fs\in \ITw$, we have \[\sLL(\fs):= \Li_{\fs}(\bone)=\zeta_{A}(\fs) =:\sLZ(\fs),\]
whence obtaining $\cZ_{w}^{\Li}= \cZ_{w}^{\zeta}$. 
\end{proof}

\begin{theorem} \label{theorem-generator}
The $k$-vector space $\cZ_{w}$ is spanned by the elements $\Li_{\fs}(\bone)$ for $\fs \in \INDw$.
\end{theorem}

\begin{proof}
We set $d_{w}' := |\ITw| = |\INDw|$.
By the definition of $\sU^{\Li}$ and Theorem \ref{theorem-algo}~(2), there exists $e \gg 0$ such that $(\sU^{\Li})^{e}(\cH_{w / \Fp[L_{1}]}) \subset \cH^{\rT}_{w / \Fp[L_{1}]}$, where $\cH_{w / \Fp[L_{1}]}$ (resp.\ $\cH^{\rT}_{w / \Fp[L_{1}]}$) is the $\Fp[L_{1}]$-module spanned by $\fs \in \cI_{w}$ (resp.\ $\fs \in \IT_{w}$) in $\cH_{w}$. We note that according to Remark~\ref{Rem:Us=s}, $(\sU^{\Li})^{e}$ is independent of $e$ for $e\gg 0$.

To show the desired result, it suffices to prove that the determinant of the matrix \[U = (u_{\fs, \fn})_{\fs \in \INDw, \fn \in \ITw} \in \Mat_{d_{w}}(\Fp[L_{1}])\] arising from
\begin{align*}
(\sU^{\Li})^{e}(\fs) = \sum_{\fn \in \ITw} u_{\fs, \fn} [\fn] \ \ (\fs \in \INDw)
\end{align*}
is non-zero. We mention that $\INDw$ is equal to $\ITw$ when $w \leq q$, and in this case the result is valid. So, in what follows, we assume that $\INDw \neq \ITw$.

Taking any $\fs = (s_{1}, \ldots, s_{r}) \in \INDw$ but $\fs \notin \ITw$, based on~~\eqref{E:s^T} and the definition of $\INDw$ we  write $\fs = (\fs^{\rT}, q^{\{0\}}, \fs')$  with $\fs^{\rT} \in \IT$ ($q^{\left\{ 0\right\}}:=\emptyset$), and $\fs' \neq \emptyset$. Note that $\fs'$ is of the form: $\fs' = (s'_{1}, \ldots)$ with $s'_{1} > q$.
By definition~\eqref{E:DLi}, $D^{\Li}_{q, \fs''} = 0$.
Thus
\begin{align}\label{E:sULis}
\sU^{\Li}(\fs) = - [\fs^{\rT}, q, \fs''] + L_{1} [\fs^{\rT}, \alpha_{q}^{\Li}(\fs'')] + L_{1} (\fs^{\rT} \boxplus \alpha_{q}^{\Li}(\fs'')),
\end{align}
where $\fs'' := (s'_{1} - q, \fs'_{-})$.
Thus for each $\fs = (m_{1} q + n_{1}, \ldots, m_{r} q + n_{r}) \in \INDw$ with  $m_{i} \geq 0$ ($m_{i}$ are not all zero since $\fs\notin \ITw$) and $1 \leq n_{i} \leq q - 1$, we have
\begin{align*}
(\sU^{\Li})^{e}(\fs) &= (\sU^{\Li})^{e + m_{1} + \cdots + m_{r}}(\fs) \\
&  \in (\sU^{\Li})^{e}\left((- 1)^{m_{1} + \cdots + m_{r}} [q^{\{ m_{1} \}}, n_{1}, \ldots, q^{\{ m_{r} \}}, n_{r}] + L_{1} \cdot \cH_{w / \Fp[L_{1}]}\right)  \\
&\subset (- 1)^{m_{1} + \cdots + m_{r}} [q^{\{ m_{1} \}}, n_{1}, \ldots, q^{\{ m_{r} \}}, n_{r}] + L_{1} \cdot \cH^{\rT}_{w / \Fp[L_{1}]}.
\end{align*} Note that the first identity above comes from  the fact that $(\sU^{\Li})^{e}(\fs)\in \cH^{\rT}_{w / \Fp[L_{1}]}$, which is fixed by $(\sU^{\Li})^{m_{1} + \cdots + m_{r}}$ by Remark~\ref{Rem:Us=s}. The second belonging follows from~\eqref{E:sULis},
whence we have the last inclusion because of $(q^{\{ m_{1} \}}, n_{1}, \ldots, q^{\{ m_{r} \}}, n_{r}) \in \ITw$ and $ (\sU^{\Li})^{e} (\cH_{w / \Fp[L_{1}]}) \subset \cH^{\rT}_{w / \Fp[L_{1}]}$.
Since \[(m_{1} q + n_{1}, \ldots, m_{r} q + n_{r}) \mapsto (q^{\{ m_{1} \}}, n_{1}, \ldots, q^{\{ m_{r} \}}, n_{r})\] gives a bijection between $\INDw$ and $\ITw$, we have $\det(U \bmod L_{1}) = \pm 1$ in $\Fp$.
In particular, $\det(U) \neq 0$.
\end{proof}

\section{Frobenius difference equations}\label{Sec:FDE}
In this section, we focus on a specific system~\eqref{eq-Frob} of Frobenius difference equations arising from our study of MZV's, and the major result is Theorem~\ref{theorem-dimension} giving the precise dimension of the solution space of\eqref{eq-Frob}.

\subsection{ABP-criterion}
Let $t$ be a new variable, and $\laurent{\CC_{\infty}}{t}$ be the field of Laurent series in $t$ over $\CC_{\infty}$. For any integer $n$, we define the following {\it{$n$-fold Frobenious twisting}}:
\[\tau^{n}:= \left(f\mapsto f^{(n)} \right): \laurent{\CC_{\infty}}{t} \rightarrow \laurent{\CC_{\infty}}{t},  \]
where $f^{(n)}:=\sum a_{i}^{q^{n}}t^{i}$ for $f=\sum a_{i}t^{i}\in \laurent{\CC_{\infty}}{t}$. We further extend $\tau^{n}$ to the action on  $\Mat_{r\times s}(\laurent{\CC_{\infty}}{t})$ by entrywise action.

Note that as an automorphism of the field $\laurent{\CC_{\infty}}{t}$, $\tau^{n}$ stabilizes several subrings such as the power series ring $\power{\CC_{\infty}}{t}$, the Tate algebra 
\[\TT:=\left\{f\in \power{\CC_{\infty}}{t}| f\hbox{ converges on }|t|_{\infty}\leq 1 \right\},\] and the polynomial rings $\CC_{\infty}[t], \ok[t]$ as well as the subfields $\CC_{\infty}(t)$ and $\ok(t)$.  When $n=-1$, we denote by $\sigma := \tau^{-1}$. For $n\in \left\{1,-1 \right\}$, the following fixed elements are particularly used in this paper:
\[ \TT^{\tau}=\ok[t]^{\tau}=\FF_{q}[t]=\TT^{\sigma}=\ok[t]^{\sigma} \hbox{ and }\ok(t)^{\tau}=\FF_{q}(t)=\ok(t)^{\sigma}. \]

Following~\cite{ABP04}, we denote by $\cE\subset  \power{\CC_{\infty}}{t}$ the subring consisting of power series $f=\sum_{i=0}^{\infty}a_{i}t^{i}$ with algebraic coefficients ($a_{i}\in\ok$, $\forall i\in \ZZ_{\geq 0}$) for which 
\[\lim_{i\rightarrow \infty} \sqrt[i]{|a_{i}|_{\infty}}=0\hbox{ and }[k_{\infty}\left(a_{0},a_{1},\ldots \right):k_{\infty} ]<\infty .\] It follows that such any power series in $\cE$ converges on whole $\CC_{\infty}$, and we have \[ \cE^{\tau}=\FF_{q}[t]=\cE^{\sigma}.\] The primary tool that we use in this paper to show linear independence results is the following  ABP-criterion

\begin{theorem}[{\cite[Theorem~3.1.1]{ABP04}}] \label{T:ABP}
Let $\Phi\in \Mat_{\ell}(\ok[t])$ be a matrix satisfying that as a polynomial in $\ok[t]$, $\det \Phi$ vanishes only (if at all) at $t=\theta$. Suppose that a column vector $\psi=\psi(t)\in \Mat_{\ell \times 1}(\cE)$ satisfies the following difference equation
\[\psi^{(-1)}=\Phi \psi  .\]
We denote by $\psi(\theta)\in \Mat_{\ell \times 1}(\CC_{\infty})$ the evaluation of $\psi$ at $t=\theta$. Then for any row vector $\rho\in \Mat_{1\times \ell}(\ok)$ for which \[\rho \psi(\theta)=0,\]  there exists a row vector $\matheur{P}\in \Mat_{1\times \ell}(\ok[t])$ such that 
\[ \matheur{P}\psi=0,\hbox{ and }\matheur{P}(\theta)=\rho .\]
\end{theorem}

The spirit of ABP-criterion asserts any $\ok$-linear relation among the entries of $\psi(\theta)$ can be lifted to a $\ok[t]$-linear relation among the entries of $\psi$. We mention that the condition on $\Phi$ in the theorem above arises from {\it{dual $t$-motives}}, and refer the reader to~\cite{ABP04, P08} and~\cite{A86} also.

\subsection{Some lemmas}

For each $\fs = (s_{1}, \ldots, s_{r}) \in \cI$ and $M \subset \cI$, we set
\begin{align*}
\sI(\fs) := \{ \emptyset, (s_{1}), (s_{1}, s_{2}), \ldots, (s_{1}, \ldots, s_{r}) \}
\ \ \textrm{and} \ \
\sI(M) := \bigcup_{\fs \in M} \sI(\fs).
\end{align*}
Throughout this paper, we denote by $\eR$ the following $\FF_{q}(t)$-algebra
\begin{align*}
\eR := \{ f / g \ | \ f \in \ok[t], \ g \in \Fq[t] \setminus \{ 0 \} \} \subset \ok(t).
\end{align*}
For any nonempty set $J$, it is understood that $\eR^{J}$ refers to the $\FF_{q}(t)$-vector space  $\bigoplus_{x\in J} \eR$ and any element of $\eR^{J}$ is written in the form $(r_x)_{x\in J}$ with $r_{x}\in \eR$.

Given any nonempty subset $M \subset \cI_{w}$ for some $w \geq 0$, we consider the system of Frobenius equations
\begin{align} \label{eq-Frob}
\varepsilon_{\fs}^{(1)} = \varepsilon_{\fs} (t - \theta)^{w - \wt(\fs)} + \sum_{\substack{s' > 0 \\ (\fs, s') \in \sI(M)}} \varepsilon_{(\fs, s')} (t - \theta)^{w - \wt(\fs)} \ \ (\fs \in \sI(M)) \tag{E$_{M}$}
\end{align}
with  $(\varepsilon_{\fs})_{\fs \in \sI(M)} \in \eR^{\sI(M)}$. 
We mention that~\eqref{eq-Frob} is basically the case of $Q_i=1$ from~\cite[(3.1.3)]{CPY19}, which is derived from the study of establishing a criterion for the {\it{zeta-like}} MZV's.
Note that our $\varepsilon_{\fs}$ is essentially referring  to $\delta_i^{(-1)}$ used in~\cite[(3.1.3)]{CPY19} (with $Q_{i}=1$  in~\cite{CPY19}), but we follow~\cite{ND21} to express  \eqref{eq-Frob} with subscripts parameterized by indices. See also~\cite[(6.3), (6.4)]{ND21} in the depth two case.

\begin{remark}
Note that if $w=0$, we have $M=\cI_{0}=\left\{\emptyset\right\}$ and so $\sI(M)=\left\{ \emptyset\right\}$. In this special case, the equation~\eqref{eq-Frob} becomes $\varepsilon_{\emptyset}^{(1)}=\varepsilon_{\emptyset}$.
\end{remark}

\begin{lemma} \label{lem-in-A}
Let $F = \sum_{i = 0}^{n} f_{i} t^{i}, G \in A[t]$ be polynomials over $A$ such that $n = \deg_{t} F \geq 1$ and $f_{n} \in \Fq^{\times}$.
If $\varepsilon \in \eR$ satisfies $\varepsilon^{(1)} = \varepsilon F + G$, then we have $\varepsilon \in \Fq(t)[\theta]$.
\end{lemma}

\begin{proof}
%\red{(cf.~[Chen, Kuan-Lin]?)}
    Our method is inspired by H.-J. Chen's formulation in the proof of \cite[Thm.~2 (a)]{KL16} (see Step I of the proof of \cite[Thm.~6.1.1]{C16} also).
By multiplying the denominator of $\varepsilon$ on the equation, without loss of generality
 we assume that $\varepsilon \in \ok[t] \setminus \{ 0 \}$.
We put $m := \deg_{t} \varepsilon \geq 0$.
Since $\deg_{t} \varepsilon^{(1)} = \deg_{t} \varepsilon < \deg_{t} (\varepsilon F)$, we have $\deg_{t} G = m + n$.
We set
\begin{align*}
\varepsilon = \sum_{i = 0}^{m} a_{i} t^{i} \ (a_{i} \in \ok) \ \ \textrm{and} \ \ G = \sum_{i = 0}^{m + n} g_{i} t^{i} \ (g_{i} \in A).
\end{align*}
Moreover, we set $a_{i} := 0$ (resp.\ $f_{i} := 0$) when $i$ does not satisfy $0 \leq i \leq m$ (resp.\ $0 \leq i \leq n$).
Then we have
\begin{align*}
\sum_{i = 0}^{m} a_{i}^{q} t^{i} = \sum_{j = 0}^{m} a_{j} t^{j} \sum_{\ell = 0}^{n} f_{\ell} t^{\ell} + \sum_{i = 0}^{m + n} g_{i} t^{i}
= \sum_{i = 0}^{m + n} \sum_{j + \ell = i} a_{j} f_{\ell} t^{i} + \sum_{i = 0}^{m + n} g_{i} t^{i}.
\end{align*}
By comparing the coefficients of $t^{i}$ for $m + 1 \leq i \leq m + n$, we have
\begin{align*}
a_{i - n} f_{n} = - g_{i} - \sum_{j > i - n} a_{j} f_{i - j}.
\end{align*}
Then we can show that $a_{m}, a_{m - 1}, \ldots, a_{m + 1 - n} \in A$ inductively.

Suppose that $m + 1 - n \geq 1$ $(\Leftrightarrow m \geq n)$.
Then by comparing the coefficients of $t^{i}$ for $n \leq i \leq m$, we have
\begin{align*}
a_{i - n} f_{n} = a_{i}^{q} - g_{i} - \sum_{j > i - n} a_{j} f_{i - j}.
\end{align*}
Then we can show that $a_{m - n}, a_{m - n - 1}, \ldots, a_{0} \in A$ inductively.
Therefore, in any case we have $\varepsilon \in A[t]$.
\end{proof}

%\red{(Mention~\cite{C16} for the following bound?)}
The next lemma can be viewed as an improvement of Step II \cite[Thm.~6.1.1]{C16} and \cite[Thm.~2 (b)]{KL16} which gives a better upper bound for the degree of solutions $(\varepsilon_{\fs})_{\fs \in \sI(M)} \in \eR^{\sI(M)}$ satisfying \eqref{eq-Frob}.

\begin{lemma} \label{lem-degree}
Let $w \geq 0$ and $\emptyset\neq M \subset \cI_{w}$.
If $(\varepsilon_{\fs})_{\fs \in \sI(M)} \in \eR^{\sI(M)}$ satisfies \eqref{eq-Frob}
\begin{align*}
\varepsilon_{\fs}^{(1)} = \varepsilon_{\fs} (t - \theta)^{w - \wt(\fs)} + \sum_{\substack{s' > 0 \\ (\fs, s') \in \sI(M)}} \varepsilon_{(\fs, s')} (t - \theta)^{w - \wt(\fs)} \ \ (\fs \in \sI(M)),
\end{align*}
then we have $\varepsilon_{\fs} \in \Fq(t)[\theta]$ and $\deg_{\theta} \varepsilon_{\fs} \leq \dfrac{w - \wt(\fs)}{q - 1}$ for each $\fs \in \sI(M)$.
\end{lemma}

\begin{proof}
%\red{(cf.~[C16, Kuan-Lin]?)}
When $\fs \in M$, then we have $\varepsilon_{\fs} \in \Fq(t)$ since in this case, $\varepsilon_{\fs}^{(1)}=\varepsilon_{\fs}$.
Thus the statements are clearly valid.
By Lemma \ref{lem-in-A} and the induction on $w - \wt(\fs)$, we have $\varepsilon_{\fs} \in \Fq(t)[\theta]$ for all $\fs \in \sI(M)$.
We now consider $\fs \in \sI(M) \setminus M$ and suppose that $\deg_{\theta} \varepsilon_{(\fs, s')} \leq \dfrac{w - \wt(\fs, s)}{q - 1}$ for all $s' > 0$ with $(\fs, s') \in \sI(M)$.

Suppose on the contrary that $\deg_{\theta} \varepsilon_{\fs} > \dfrac{w - \wt(\fs)}{q - 1}$.
Then we have $\deg_{\theta} \varepsilon_{\fs}^{(1)} > \dfrac{q (w - \wt(\fs))}{q - 1}$.
On the other hand, the induction hypothesis implies
\begin{align*}
\deg_{\theta} (\varepsilon_{(\fs, s')} (t - \theta)^{w - \wt(\fs)}) \leq \dfrac{w - \wt(\fs, s')}{q - 1} + w - \wt(\fs) < \dfrac{q (w - \wt(\fs))}{q - 1}.
\end{align*}
It follows that $\deg_{\theta} \varepsilon_{\fs}^{(1)} = \deg_{\theta} (\varepsilon_{\fs} (t - \theta)^{w - \wt(\fs)})$, and hence $\deg_{\theta} \varepsilon_{\fs} = \dfrac{w - \wt(\fs)}{q - 1}$,
which is a contradiction.
\end{proof}

\subsection{Dimension of $\sX_{w}$} 
For $w\in \ZZ_{w\geq 0}$, We define $\INDzw \subset \INDw$ as follows:
 \begin{equation}\label{E:INDzw}
 \INDzw = \left\{ \begin{array}{@{}ll} \left\{\emptyset\right\} & \textrm{if} \ w=0 \\ \left\{ (s_{1},\ldots,s_{r})\in \INDw|s_{2},\ldots,s_{r} \hbox{ are divisible by }q-1  \right\} & \textrm{if} \ w>0 \end{array} \right.
  \end{equation}
Note that the description of indices in $\INDzw$ comes from the simultaneously Eulerian phenomenon in~\cite[Cor.~4.2.3]{CPY19}.
For each $w \geq 0$, we consider the system of Frobenius equations \eqref{eq-Frob} for $M = \INDzw$:
\begin{align} \label{eq-Frob-w}
\varepsilon_{\fs}^{(1)} = \varepsilon_{\fs} (t - \theta)^{w - \wt(\fs)} + \sum_{\substack{s' > 0 \\ (\fs, s') \in \sI(\INDzw)}} \varepsilon_{(\fs, s')} (t - \theta)^{w - \wt(\fs)} \ \ (\fs \in \sI(\INDzw)) \tag{E$_{w}$}
\end{align}
with $(\varepsilon_{\fs})_{\fs \in \sI(\INDzw)} \in \eR^{\sI(\INDzw)}$. We emphasize that $(E_{w}) = (E_{\INDzw})$ and $(E_{w}) \neq (E_{\{w\}})$.

\begin{example}
    We give some explicit examples for the system of Frobenius equations \eqref{eq-Frob-w}. The simplest example is the case  $w=0$. We note that $\INDz_0=\sI(\INDz_0)=\{\emptyset\}$. Thus, \eqref{eq-Frob-w} consists of a single equation
    \[
        \varepsilon_{\emptyset}^{(1)} = \varepsilon_{\emptyset}.
    \]
    
    The second example is the case $w=q$. In this case, $\INDz_{q}=\{(1,q-1)\}$ and
    \begin{align*}
        \sI(\INDz_{q})=\{\emptyset,(1),(1,q-1)\}.
    \end{align*}
    Thus, \eqref{eq-Frob-w} consists of three equations
    \begin{align*}
        \left\{ 
            \begin{array}{@{}ll} 
                \varepsilon_{(1,q-1)}^{(1)} &= \varepsilon_{(1,q-1)}\\
                \varepsilon_{(1)}^{(1)} &= \varepsilon_{(1)} (t - \theta)^{q-1}+\varepsilon_{(1,q-1)}(t-\theta)^{q-1}\\
                \varepsilon_{\emptyset}^{(1)} &= \varepsilon_{\emptyset} (t - \theta)^{q}+\varepsilon_{(1)}(t-\theta)^{q}.
            \end{array} 
        \right.
    \end{align*}
    
    The third example is the case $w=2q-2$ with $q \geq 3$. Since $\INDz_{2q-2}=\{(q-1,q-1),(2q-2)\}$, we have
    \[
        \sI(\INDz_{2q-2})=\{\emptyset,(q-1),(q-1,q-1),(2q-2)\}.
    \]
    Then \eqref{eq-Frob-w} consists of four equations
    \begin{align*}
        \left\{ 
            \begin{array}{@{}ll} 
                \varepsilon_{(2q-2)}^{(1)} &= \varepsilon_{(2q-2)}\\
                \varepsilon_{(q-1,q-1)}^{(1)} &= \varepsilon_{(q-1,q-1)}\\
                \varepsilon_{(q-1)}^{(1)} &= \varepsilon_{(q-1)} (t - \theta)^{q-1}+\varepsilon_{(q-1,q-1)}(t-\theta)^{q-1}\\
                \varepsilon_{\emptyset}^{(1)} &= \varepsilon_{\emptyset} (t - \theta)^{2q-2}+\varepsilon_{(q-1)}(t-\theta)^{2q-2}+\varepsilon_{(2q-2)}(t-\theta)^{2q-2}.
            \end{array} 
        \right.
    \end{align*}
\end{example}

\begin{definition} For each $w\in \ZZ_{\geq 0}$, we let $\sX_{w}$ be the set of $\eR$-valued solutions  of \eqref{eq-Frob-w}.
\end{definition}

\begin{remark}

Since $f^{(1)}=f$ for any $f\in \FF_{q}(t)\subset \eR$, we
see that $\sX_{w}$ forms an $\Fq(t)$-vector subspace of $\eR^{\sI(\INDzw)}$.
\end{remark}

The main result of this section is the following dimension formula for $\sX_{w}$.

\begin{theorem} \label{theorem-dimension} For any $w\in \ZZ_{\geq 0}$,
we have \[\dim_{\Fq(t)} \sX_{w} = \left\{ \begin{array}{@{}ll} 1 & \textrm{if} \ (q - 1) \mid w \\ 0 & \textrm{if} \ (q - 1) \nmid w \end{array} \right.\]
When $w = \ell (q - 1)$ for some $\ell\in \ZZ_{\geq 0}$, we can choose a generator $(\eta_{\ell; \fs})_{\fs \in \sI(\INDzw)} \in \sX_{w}$ such that
\begin{align*}
\eta_{\ell; \emptyset} = \sum_{i = 0}^{\ell} b_{\ell i} (t - \theta)^{i}, \ \
b_{\ell i} \in T \Fp[T]_{(T)} \ (0 \leq i < \ell) \ \ \textrm{and} \ \
b_{\ell \ell} = 1,
\end{align*}
where $T := t - t^{q}$ and $\Fp[T]_{(T)} \subset \Fp(t)$ is the localization of $\Fp[T]$ at the prime ideal $(T)$.
\end{theorem}

\begin{proof}
Note that by Lemma \ref{lem-degree}, we have $\sX_{w} \subset \Fq(t)[\theta]^{\sI(\INDzw)}$.
We write
\begin{align*}
w &= \ell (q - 1) + s \ \ (\ell \geq 0, \ \ 0 \leq s < q - 1), \\
\ell &= m q + n \ \ (m \geq 0, \ \ 0 \leq n < q).
\end{align*}

First we assume that $s = 0$ and hence $w = \ell (q - 1)$ with $\ell \geq 0$.
When $\ell = 0$, we have $\sI(\cI^{\mathrm{ND}_{0}}_{0}) = \{ \emptyset \}$ and the equation is $\varepsilon_{\emptyset}^{(1)} = \varepsilon_{\emptyset}$.
Thus $\sX_0=\Fq(t)$, and we have a generator $\eta_{0; \emptyset} = 1$.
Let $\ell \geq 1$ and suppose that the desired properties hold for weight $\ell' (q - 1)$ with $\ell' < \ell$.
The system of Frobenius equations (E$_{\ell (q - 1)}$) becomes
\begin{align} \label{eq-Fr-empty-even}
\varepsilon_{\emptyset}^{(1)} = \varepsilon_{\emptyset} (t - \theta)^{\ell (q - 1)} + \sum_{\substack{0 \leq j < \ell \\ j \not\equiv \ell \bmod q}} \varepsilon_{((\ell - j) (q - 1))}(t - \theta)^{\ell (q - 1)}
\end{align}
and
\begin{align*}
\varepsilon_{((\ell - j) (q - 1), \fs)}^{(1)} &= \varepsilon_{((\ell - j) (q - 1), \fs)} (t - \theta)^{j (q - 1) - \wt(\fs)} \\
& \ \ \ + \sum_{\substack{s' > 0 \\ (\fs, s') \in \sI(\cI^{\mathrm{ND}_{0}}_{j (q - 1)})}} \varepsilon_{((\ell - j) (q - 1), \fs, s')} (t - \theta)^{j (q - 1) - \wt(\fs)}
\end{align*}
for $0 \leq j < \ell$ with $j \not\equiv \ell \bmod q$ and $\fs \in \sI(\cI^{\mathrm{ND}_{0}}_{j (q - 1)})$.
Since $(\varepsilon_{((\ell - j) (q - 1), \fs)})_{\fs \in \sI(\cI^{\mathrm{ND}_{0}}_{j (q - 1)})} \in \sX_{j (q - 1)}$ for each $j$, the induction hypothesis implies that we have
\begin{align*}
\varepsilon_{((\ell - j) (q - 1))} = f_{j} \sum_{i = 0}^{j} b_{j i} (t - \theta)^{i}, \ \
f_{j} \in \Fq(t), \ \
b_{j i} \in T \Fp[T]_{(T)}
\ \ \textrm{and} \ \
b_{j j} = 1.
\end{align*}
By Lemma \ref{lem-degree}, we have $\deg_{\theta} \varepsilon_{\emptyset} \leq \ell$.
We write $\varepsilon_{\emptyset} = \sum_{i = 0}^{\ell} a_{i} (t - \theta)^{i}$ with $a_{i} \in \Fq(t)$, and plug it into \eqref{eq-Fr-empty-even}, then we obtain
\begin{align*}
\sum_{j = 0}^{\ell} a_{j} (t - \theta^{q})^{j} = \sum_{i = 0}^{\ell} a_{i} (t - \theta)^{\ell (q - 1) + i} + \sum_{\substack{0 \leq j < \ell \\ j \not\equiv \ell \bmod q}} f_{j} \sum_{i = 0}^{j} b_{j i} (t - \theta)^{\ell (q - 1) + i}.
\end{align*}
Note that the equation above can be expressed as
\begin{align*}
\sum_{i = 0}^{\ell - 1} \sum_{j = i}^{\ell - 1} \binom{j}{i} T^{j - i} a_{j} (t - \theta)^{i q} &- \sum_{i = 0}^{\ell - 1} a_{i} (t - \theta)^{\ell (q - 1) + i} - \sum_{i = 0}^{\ell - 1} \sum_{\substack{i \leq j < \ell \\ j \not\equiv \ell \bmod q}} b_{j i} f_{j} (t - \theta)^{\ell (q - 1) + i} \\
&= a_{\ell} \sum_{i = 0}^{\ell - 1} \binom{\ell}{i} T^{\ell - i} (t - \theta)^{i q}.
\end{align*}
Note further that
\begin{align*}
i q < w = (\ell - m) q - n \ \Longleftrightarrow \ i < \ell - m - \dfrac{n}{q}.
\end{align*}
By comparing the coefficients of $(t - \theta)^{\nu}$ for $\nu = i q$ $(0 \leq i < \ell - m)$ or $\ell (q - 1) \leq \nu < \ell q$, this gives a system of linear equations with $(2 \ell - m)$-equations and $(2 \ell - m + 1)$-variables.
We write $\ba := (a_{0}, \ldots, a_{\ell - 1})^{\tr}$ and $\bff := (f_{0}, \ldots, f_{\ell - 1})^{\tr}$ (excluding $f_{n}, f_{q + n}, \ldots, f_{(m - 1) q + n}$).
Then the system mentioned above can be expressed as
\begin{align}\label{E:matrx U}
U \left( \begin{array}{@{}c@{}} \ba \\ \bff \end{array} \right) = a_{\ell} \bb, \ \ U \in \Mat_{2 \ell - m}(\Fp[T]_{(T)}) \ \ \textrm{and} \ \ \bb \in (T \Fp[T]_{(T)})^{2 \ell - m}.
\end{align}

As our goal of this case $w=\ell (q-1)$ is to show $\dim_{\FF_{q}(t)}\sX_{w}=1$, it is to show that the existence of $\begin{pmatrix}
\ba\\
\bff
\end{pmatrix}$ is unique up to $\FF_{q}(t)$-scalar multiple.
Therefore, from~\eqref{E:matrx U} it suffices to show that $\det(U \bmod T) \neq 0$ in $\Fp$.
Indeed, we have
\begin{align*}
&(U \bmod T) + \left( \begin{array}{@{}cc@{}} O & O \\ \Id_{\ell} & O \end{array} \right) \\
&= \begin{array}{rc@{\hspace{0.3em}}c@{\hspace{0.3em}}c@{\hspace{0.3em}}c@{\hspace{0.3em}}c@{\hspace{0.3em}}c@{\hspace{0.3em}}c@{\hspace{0.3em}}|@{\hspace{0.3em}}c@{\hspace{0.3em}}c@{\hspace{0.3em}}c@{\hspace{0.3em}}c@{\hspace{0.3em}}c@{\hspace{0.3em}}c@{\hspace{0.3em}}c@{\hspace{0.3em}}c@{\hspace{0.3em}}c@{\hspace{0.3em}}c@{\hspace{0.3em}}c@{\hspace{0.3em}}c@{\hspace{0.3em}}c@{\hspace{0.3em}}c@{\hspace{0.3em}}c@{\hspace{0.3em}}c@{}ll}
& \multicolumn{7}{c}{\overbrace{\hspace{6.0em}}^{\ell}} & \multicolumn{16}{c}{\overbrace{\hspace{20.5em}}^{\ell - m}} & & \\
\ldelim( {24}{4pt}[] & 1 & & & & & & & & & & & & & & & & & & & & & & & \rdelim) {24}{2pt}[] & \rdelim\}{3}{10pt}[{\footnotesize $\ell - m$}] \\
& & \ddots & & & & & & & & & & & & & & & & & & & & & & & \\
& & & 1 & & & & & & & & & & & & & & & & & & & & & & \\
\cline{2-22}
& & & & & & & & - 1 & & \multicolumn{1}{c@{}:}{} & & & & & & & & & & & & & & & \rdelim\}{3}{10pt}[{\footnotesize $n$}] \\
& & & & & & & & & \ddots & \multicolumn{1}{c@{}:}{} & & & & & & & & & & & & & & & \\
& & & & & & & & & & \multicolumn{1}{c@{}:}{- 1} & & & & & & & & & & & & & & & \\
\cdashline{2-22}
& & & & 1 & & & & & & \multicolumn{1}{c@{}:}{} & 0 & & & & & & & & & & & & & & \\
\cdashline{2-22}
& & & & & & & & & & & \multicolumn{1}{:c@{}}{- 1} & & \multicolumn{1}{c@{}:}{} & & & & & & & & & & & & \rdelim\}{3}{10pt}[{\footnotesize $q - 1$}] \\
& & & & & & & & & & & \multicolumn{1}{:c@{}}{} & \ddots & \multicolumn{1}{c@{}:}{} & & & & & & & & & & & & \\
& & & & & & & & & & & \multicolumn{1}{:c@{}}{} & & \multicolumn{1}{c@{}:}{- 1} & & & & & & & & & & & & \\
\cdashline{2-22}
& & & & & 1 & & & & & & & & \multicolumn{1}{c@{}:}{} & 0 & & & & & & & & & & & \\
\cdashline{2-22}
& & & & & & & & & & & & & & \ddots & & & & & & & & & & & \\
& & & & & & \ddots & & & & & & & & & \ddots & & & & & & & & & & \vdots \\
& & & & & & & & & & & & & & & & \multicolumn{1}{c@{}}{\ddots} & & & & & & & & & \\
\cdashline{2-22}
& & & & & & & & & & & & & & & & \multicolumn{1}{:c@{}}{- 1} & & \multicolumn{1}{c@{}:}{} & & & & & & & \rdelim\}{3}{10pt}[{\footnotesize $q - 1$}] \\
& & & & & & & & & & & & & & & & \multicolumn{1}{:c@{}}{} & \ddots & \multicolumn{1}{c@{}:}{} & & & & & & & \\
& & & & & & & & & & & & & & & & \multicolumn{1}{:c@{}}{} & & \multicolumn{1}{c@{}:}{- 1} & & & & & & & \\
\cdashline{2-22}
& & & & & & & 1 & & & & & & & & & & & \multicolumn{1}{c@{}:}{} & 0 & & & & & & \\
\cdashline{2-22}
& & & & & & & & & & & & & & & & & & & \multicolumn{1}{:c@{}}{- 1} & & & & & & \rdelim\}{3}{10pt}[{\footnotesize $q - 1$}] \\
& & & & & & & & & & & & & & & & & & & \multicolumn{1}{:c@{}}{} & \ddots & & & & & \\
& & & & & & & & & & & & & & & & & & & \multicolumn{1}{:c@{}}{} & & - 1 & & & &
\end{array}.
\end{align*}
The matrix $\left(U \bmod T\right)+ \left( \begin{array}{@{}cc@{}} O & O \\ \Id_{\ell} & O \end{array} \right)$ can be described as follows:
\begin{itemize}
\item [(i)] We first set up the lower right submatrix  by beginning with a matrix of diagonal blocks, the first one being $-\Id_n$ and the remaining $m$ of them being equal to $-\Id_{q-1}$.  Then we insert zero rows between the blocks to obtain our submatrix of size $\ell \times (\ell -m)$.

\item [(ii)] We then put the identity matrix $\Id_{\ell-m}$ on the upper left corner. 

\item [(iii)] Finally we define the $\left((\ell-m)+1\right)$st column to the $\ell$th column, each of length $2\ell -m$. In the $\left((\ell-m)+i\right)$th column, all entries are zero except for the entry $1$ occurring in the same row as the $i$th row of zeroes inserted into the matrix constructed in step (i), i.e. the row of zeroes lying above the $i$th diagonal block equal to $-\Id_{q-1}$  in step (i).  Thus we have our $2(\ell-m)\times2(\ell-m)$ matrix $\left(U \bmod T\right)+ \left( \begin{array}{@{}cc@{}} O & O \\ \Id_{\ell} & O \end{array} \right)$.
\end{itemize}
For example,
let $q = 3$, $w = 14 = 7 (3 - 1)$ \ ($\ell = 7 = 2 \cdot 3 + 1$, $m = 2$, $n = 1$). Then we have
\begin{align*}
U \bmod T =
\left( \begin{array}{@{}ccccccc|ccccc@{}}
1 & 0 & 0 & 0 & 0 & 0 & 0 & 0 & 0 & 0 & 0 & 0 \\
0 & 1 & 0 & 0 & 0 & 0 & 0 & 0 & 0 & 0 & 0 & 0 \\
0 & 0 & 1 & 0 & 0 & 0 & 0 & 0 & 0 & 0 & 0 & 0 \\
0 & 0 & 0 & 1 & 0 & 0 & 0 & 0 & 0 & 0 & 0 & 0 \\
0 & 0 & 0 & 0 & 1 & 0 & 0 & 0 & 0 & 0 & 0 & 0 \\ \hline
-1 & 0 & 0 & 0 & 0 & 0 & 0 & -1 & 0 & 0 & 0 & 0 \\
0 & -1 & 0 & 0 & 0 & 1 & 0 & 0 & 0 & 0 & 0 & 0 \\
0 & 0 & -1 & 0 & 0 & 0 & 0 & 0 & -1 & 0 & 0 & 0 \\
0 & 0 & 0 & -1 & 0 & 0 & 0 & 0 & 0 & -1 & 0 & 0 \\
0 & 0 & 0 & 0 & -1 & 0 & 1 & 0 & 0 & 0 & 0 & 0 \\
0 & 0 & 0 & 0 & 0 & -1 & 0 & 0 & 0 & 0 & -1 & 0 \\
0 & 0 & 0 & 0 & 0 & 0 & -1 & 0 & 0 & 0 & 0 & -1 \\
\end{array} \right)
\end{align*}

To evaluate $\det \left(U \bmod T\right)$ we appeal to elementary row operations which do not change the value of the determinant. Adding each of the top $\ell-m$ rows to the corresponding rows in the lower left matrix kills all entries there coming from $-\Id_{\ell}$ except the $m$ bottom ones, which lie in the $((\ell-m)+1)$st one to the $\ell$th column.  But each $-1$ in these last $m$ rows 
can be killed by adding the appropriate row from step (iii) which has all zero entries except $1$ in a unique column between $\ell -m +1$ and $ \ell$ .  Doing this leaves only the non-zero entries in those $m$ columns the ones inserted in step (iii).  We end up with a matrix which has a unique $\pm 1$ in each row and column.  So the elementary product expansion evaluates the determinant as $\pm 1$. Note that one can also use elementary column operations to show the desired identity.

Next we assume that $1 \leq s < q - 1$ and put $w := \ell (q - 1) + s$.
The system of Frobenius equations (E$_{w}$) is
\begin{align} \label{eq-Fr-empty-odd}
\varepsilon_{\emptyset}^{(1)} = \varepsilon_{\emptyset} (t - \theta)^{w} + \sum_{\substack{0 \leq j \leq \ell \\ j \not\equiv \ell - s \bmod q}} \varepsilon_{((\ell - j) (q - 1) + s)}(t - \theta)^{w}
\end{align}
and
\begin{align*}
\varepsilon_{((\ell - j) (q - 1) + s, \fs)}^{(1)} &= \varepsilon_{((\ell - j) (q - 1) + s, \fs)} (t - \theta)^{j (q - 1) - \wt(\fs)} \\
& \ \ \ + \sum_{\substack{s' > 0 \\ (\fs, s') \in \sI(\cI^{\mathrm{ND}_{0}}_{j (q - 1)})}} \varepsilon_{((\ell - j) (q - 1) + s, \fs, s')} (t - \theta)^{j (q - 1) - \wt(\fs)}
\end{align*}
for $0 \leq j \leq \ell$ with $j \not\equiv \ell - s \bmod q$ and $\fs \in \sI(\cI^{\mathrm{ND}_{0}}_{j (q - 1)})$.
Since $(\varepsilon_{((\ell - j) (q - 1) + s, \fs)})_{\fs \in \sI(\cI^{\mathrm{ND}_{0}}_{j (q - 1)})} \in \sX_{j (q - 1)}$ for each $j$, we have
\begin{align*}
\varepsilon_{((\ell - j) (q - 1) + s)} = f_{j} \sum_{i = 0}^{j} b_{j i} (t - \theta)^{i}, \ \
f_{j} \in \Fq(t), \ \
b_{j i} \in T \Fp[T]_{(T)}
\ \ \textrm{and} \ \
b_{j j} = 1.
\end{align*}
By Lemma \ref{lem-degree}, we have $\deg_{\theta} \varepsilon_{\emptyset} \leq \ell$.
We write $\varepsilon_{\emptyset} = \sum_{i = 0}^{\ell} a_{i} (t - \theta)^{i}$ with $a_{i} \in \Fq(t)$.
Then \eqref{eq-Fr-empty-odd} becomes
\begin{align*}
\sum_{j = 0}^{\ell} a_{j} (t - \theta^{q})^{j} = \sum_{i = 0}^{\ell} a_{i} (t - \theta)^{w + i} + \sum_{\substack{0 \leq j \leq \ell \\ j \not\equiv \ell - s \bmod q}} f_{j} \sum_{i = 0}^{j} b_{j i} (t - \theta)^{w + i}.
\end{align*}
This identity can be written as
\begin{align*}
\sum_{i = 0}^{\ell} \sum_{j = i}^{\ell} \binom{j}{i} T^{j - i} a_{j} (t - \theta)^{i q} &- \sum_{i = 0}^{\ell} a_{i} (t - \theta)^{w + i} - \sum_{i = 0}^{\ell} \sum_{\substack{i \leq j \leq \ell \\ j \not\equiv \ell - s \bmod q}} b_{j i} f_{j} (t - \theta)^{w + i} = 0.
\end{align*}
We note that
\begin{align*}
i q < w = (\ell - m) q + (s - n) \ \Longleftrightarrow \ i < \ell - m + \dfrac{s - n}{q}.
\end{align*}

When $s \leq n$, by comparing the coefficients of $(t - \theta)^{\nu}$ for $\nu = i q$ $(0 \leq i < \ell - m)$ or $w \leq \nu \leq w + \ell$, this gives a system of linear equations with $(2 \ell - m + 1)$-equations and $(2 \ell - m + 1)$-variables.
We write $\ba := (a_{0}, \ldots, a_{\ell})^{\tr}$ and $\bff := (f_{0}, \ldots, f_{\ell})^{\tr}$ (excluding $f_{n - s}, f_{q + n - s}, \ldots, f_{\ell - s}$).
Then the system can be written as
\begin{align*}
U \left( \begin{array}{@{}c@{}} \ba \\ \bff \end{array} \right) = \mathbf{0} \ \ \textrm{and} \ \ U \in \Mat_{2 \ell - m + 1}(\Fp[T]_{(T)}).
\end{align*}
Since we aim to show that the solution space $\sX_{w}$ is trivial in this case, it suffices to show that $\det(U \bmod T) \neq 0$ in $\Fp$.
Indeed, we have
\begin{align*}
&(U \bmod T) + \left( \begin{array}{@{}cc@{}} O & O \\ \Id_{\ell + 1} & O \end{array} \right) \\
&= \begin{array}{rc@{\hspace{0.3em}}c@{\hspace{0.3em}}c@{\hspace{0.3em}}c@{\hspace{0.3em}}c@{\hspace{0.3em}}c@{\hspace{0.3em}}c@{\hspace{0.3em}}|@{\hspace{0.3em}}c@{\hspace{0.3em}}c@{\hspace{0.3em}}c@{\hspace{0.3em}}c@{\hspace{0.3em}}c@{\hspace{0.3em}}c@{\hspace{0.3em}}c@{\hspace{0.3em}}c@{\hspace{0.3em}}c@{\hspace{0.3em}}c@{\hspace{0.3em}}c@{\hspace{0.3em}}c@{\hspace{0.3em}}c@{\hspace{0.3em}}c@{\hspace{0.3em}}c@{\hspace{0.3em}}c@{}ll}
& \multicolumn{7}{c}{\overbrace{\hspace{6.0em}}^{\ell + 1}} & \multicolumn{16}{c}{\overbrace{\hspace{20.5em}}^{\ell - m}} & & \\
\ldelim( {24}{4pt}[] & 1 & & & & & & & & & & & & & & & & & & & & & & & \rdelim) {24}{2pt}[] & \rdelim\}{3}{10pt}[{\footnotesize $\ell - m$}] \\
& & \ddots & & & & & & & & & & & & & & & & & & & & & & & \\
& & & 1 & & & & & & & & & & & & & & & & & & & & & & \\
\cline{2-22}
& & & & & & & & - 1 & & \multicolumn{1}{c@{}:}{} & & & & & & & & & & & & & & & \rdelim\}{3}{10pt}[{\footnotesize $n - s$}] \\
& & & & & & & & & \ddots & \multicolumn{1}{c@{}:}{} & & & & & & & & & & & & & & & \\
& & & & & & & & & & \multicolumn{1}{c@{}:}{- 1} & & & & & & & & & & & & & & & \\
\cdashline{2-22}
& & & & 1 & & & & & & \multicolumn{1}{c@{}:}{} & 0 & & & & & & & & & & & & & & \\
\cdashline{2-22}
& & & & & & & & & & & \multicolumn{1}{:c@{}}{- 1} & & \multicolumn{1}{c@{}:}{} & & & & & & & & & & & & \rdelim\}{3}{10pt}[{\footnotesize $q - 1$}] \\
& & & & & & & & & & & \multicolumn{1}{:c@{}}{} & \ddots & \multicolumn{1}{c@{}:}{} & & & & & & & & & & & & \\
& & & & & & & & & & & \multicolumn{1}{:c@{}}{} & & \multicolumn{1}{c@{}:}{- 1} & & & & & & & & & & & & \\
\cdashline{2-22}
& & & & & 1 & & & & & & & & \multicolumn{1}{c@{}:}{} & 0 & & & & & & & & & & & \\
\cdashline{2-22}
& & & & & & & & & & & & & & \ddots & & & & & & & & & & & \\
& & & & & & \ddots & & & & & & & & & \ddots & & & & & & & & & & \vdots \\
& & & & & & & & & & & & & & & & \multicolumn{1}{c@{}}{\ddots} & & & & & & & & & \\
\cdashline{2-22}
& & & & & & & & & & & & & & & & \multicolumn{1}{:c@{}}{- 1} & & \multicolumn{1}{c@{}:}{} & & & & & & & \rdelim\}{3}{10pt}[{\footnotesize $q - 1$}] \\
& & & & & & & & & & & & & & & & \multicolumn{1}{:c@{}}{} & \ddots & \multicolumn{1}{c@{}:}{} & & & & & & & \\
& & & & & & & & & & & & & & & & \multicolumn{1}{:c@{}}{} & & \multicolumn{1}{c@{}:}{- 1} & & & & & & & \\
\cdashline{2-22}
& & & & & & & 1 & & & & & & & & & & & \multicolumn{1}{c@{}:}{} & 0 & & & & & & \\
\cdashline{2-22}
& & & & & & & & & & & & & & & & & & & \multicolumn{1}{:c@{}}{- 1} & & & & & & \rdelim\}{3}{10pt}[{\footnotesize $s$}] \\
& & & & & & & & & & & & & & & & & & & \multicolumn{1}{:c@{}}{} & \ddots & & & & & \\
& & & & & & & & & & & & & & & & & & & \multicolumn{1}{:c@{}}{} & & - 1 & & & &
\end{array}
\end{align*}
and hence $\det(U \bmod T) = \pm 1$.

When $s > n$, by comparing the coefficients of $(t - \theta)^{\nu}$ for $\nu = i q$ $(0 \leq i \leq \ell - m)$ or $w \leq \nu \leq w + \ell$, this gives a system of linear equations with $(2 \ell - m + 2)$-equations and $(2 \ell - m + 2)$-variables.
We write $\ba := (a_{0}, \ldots, a_{\ell})^{\tr}$ and $\bff := (f_{0}, \ldots, f_{\ell})^{\tr}$ (excluding $f_{q + n - s}, f_{2 q + n - s}, \ldots, f_{\ell - s}$).
Then the system can be written as
\begin{align*}
U \left( \begin{array}{@{}c@{}} \ba \\ \bff \end{array} \right) = \mathbf{0} \ \ \textrm{and} \ \ U \in \Mat_{2 \ell - m + 2}(\Fp[T]_{(T)}).
\end{align*}
It is enough to show that $\det(U \bmod T) \neq 0$ in $\Fp$.
Indeed, we have
\begin{align*}
&(U \bmod T) + \left( \begin{array}{@{}cc@{}} O & O \\ \Id_{\ell + 1} & O \end{array} \right) \\
&= \begin{array}{rc@{\hspace{0.3em}}c@{\hspace{0.3em}}c@{\hspace{0.3em}}c@{\hspace{0.3em}}c@{\hspace{0.3em}}c@{\hspace{0.3em}}c@{\hspace{0.3em}}|@{\hspace{0.3em}}c@{\hspace{0.3em}}c@{\hspace{0.3em}}c@{\hspace{0.3em}}c@{\hspace{0.3em}}c@{\hspace{0.3em}}c@{\hspace{0.3em}}c@{\hspace{0.3em}}c@{\hspace{0.3em}}c@{\hspace{0.3em}}c@{\hspace{0.3em}}c@{\hspace{0.3em}}c@{\hspace{0.3em}}c@{\hspace{0.3em}}c@{\hspace{0.3em}}c@{\hspace{0.3em}}c@{}ll}
& \multicolumn{7}{c}{\overbrace{\hspace{6.0em}}^{\ell + 1}} & \multicolumn{16}{c}{\overbrace{\hspace{20.5em}}^{\ell - m + 1}} & & \\
\ldelim( {24}{4pt}[] & 1 & & & & & & & & & & & & & & & & & & & & & & & \rdelim) {24}{2pt}[] & \rdelim\}{3}{10pt}[{\footnotesize $\ell - m + 1$}] \\
& & \ddots & & & & & & & & & & & & & & & & & & & & & & & \\
& & & 1 & & & & & & & & & & & & & & & & & & & & & & \\
\cline{2-22}
& & & & & & & & - 1 & & \multicolumn{1}{c@{}:}{} & & & & & & & & & & & & & & & \rdelim\}{3}{10pt}[{\footnotesize $q + n - s$}] \\
& & & & & & & & & \ddots & \multicolumn{1}{c@{}:}{} & & & & & & & & & & & & & & & \\
& & & & & & & & & & \multicolumn{1}{c@{}:}{- 1} & & & & & & & & & & & & & & & \\
\cdashline{2-22}
& & & & 1 & & & & & & \multicolumn{1}{c@{}:}{} & 0 & & & & & & & & & & & & & & \\
\cdashline{2-22}
& & & & & & & & & & & \multicolumn{1}{:c@{}}{- 1} & & \multicolumn{1}{c@{}:}{} & & & & & & & & & & & & \rdelim\}{3}{10pt}[{\footnotesize $q - 1$}] \\
& & & & & & & & & & & \multicolumn{1}{:c@{}}{} & \ddots & \multicolumn{1}{c@{}:}{} & & & & & & & & & & & & \\
& & & & & & & & & & & \multicolumn{1}{:c@{}}{} & & \multicolumn{1}{c@{}:}{- 1} & & & & & & & & & & & & \\
\cdashline{2-22}
& & & & & 1 & & & & & & & & \multicolumn{1}{c@{}:}{} & 0 & & & & & & & & & & & \\
\cdashline{2-22}
& & & & & & & & & & & & & & \ddots & & & & & & & & & & & \\
& & & & & & \ddots & & & & & & & & & \ddots & & & & & & & & & & \vdots \\
& & & & & & & & & & & & & & & & \multicolumn{1}{c@{}}{\ddots} & & & & & & & & & \\
\cdashline{2-22}
& & & & & & & & & & & & & & & & \multicolumn{1}{:c@{}}{- 1} & & \multicolumn{1}{c@{}:}{} & & & & & & & \rdelim\}{3}{10pt}[{\footnotesize $q - 1$}] \\
& & & & & & & & & & & & & & & & \multicolumn{1}{:c@{}}{} & \ddots & \multicolumn{1}{c@{}:}{} & & & & & & & \\
& & & & & & & & & & & & & & & & \multicolumn{1}{:c@{}}{} & & \multicolumn{1}{c@{}:}{- 1} & & & & & & & \\
\cdashline{2-22}
& & & & & & & 1 & & & & & & & & & & & \multicolumn{1}{c@{}:}{} & 0 & & & & & & \\
\cdashline{2-22}
& & & & & & & & & & & & & & & & & & & \multicolumn{1}{:c@{}}{- 1} & & & & & & \rdelim\}{3}{10pt}[{\footnotesize $s$}] \\
& & & & & & & & & & & & & & & & & & & \multicolumn{1}{:c@{}}{} & \ddots & & & & & \\
& & & & & & & & & & & & & & & & & & & \multicolumn{1}{:c@{}}{} & & - 1 & & & &
\end{array}
\end{align*}
and hence $\det(U \bmod T) = \pm 1$.
\end{proof}

\section{Linear independence}\label{Sec:Linear Indep}

    In this section, we aim to show that $\{\Li_\fs(\bone)\in  k_\infty\mid\fs\in\INDw\}$ is a $k$-linearly independent set. To begin with, we adopt the following setting. Let $(-\theta)^{\frac{1}{q-1}}\in \ok^\times$ be a fixed $(q-1)$st root of $-\theta$. Following~\cite{ABP04}, we consider the following power series
    \begin{align*}
        \Omega(t):=(-\theta)^{\frac{-q}{q-1}}\prod_{i=1}^\infty\left(1-\frac{t}{\theta^{q^i}}\right)\in\power{\overline{k_\infty}}{t}.
    \end{align*}
    Note that it satisfies the Frobenius difference equation
    \begin{align*}
        \Omega^{(-1)}(t)=(t-\theta)\Omega(t).
    \end{align*}
    In addition, we have $\Omega(t)\in\mathcal{E}$ and $\tpi:=1/\Omega(\theta)$ is the fundamental period of the Carlitz module (see~\cite[Sec.~3.1.2]{ABP04}).
    We also set $\LL_{0} := 1$ and $\LL_{d} := (t - \theta^{q}) \cdots (t - \theta^{q^{d}})$ for $d \geq 1$. Put $\cL(\emptyset) := 1$ and  for each index $\fs = (s_{1}, \ldots, s_{r}) \in\cI_{>0}$, we define the following deformation series
    \begin{align*}
        \cL(\fs) := \Omega^{\wt(\fs)} \sum_{d_{1} > \cdots > d_{r} \geq 0} \dfrac{1}{\LL_{d_{1}}^{s_{1}} \cdots \LL_{d_{r}}^{s_{r}}} \in \TT,
    \end{align*}
    which was first introduced by Papanikolas~\cite{P08} for $\dep(\fs)=1$. 
    It is shown in \cite[Lemma 5.3.1]{C14} that $\cL(\fs)\in\cE$. Moreover, by \cite[Proposition 2.3.3]{CPY19}, we have 
    \begin{align*}
        \cL(\fs)|_{t = \theta^{q^{j}}} = (\cL(\fs)|_{t = \theta})^{q^{j}} = (\Li_{\fs}(\bone) / \tpi^{\wt(\fs)})^{q^{j}}
    \end{align*}
    for all $j \geq 0$ (see also \cite[Lemma 5.3.5]{C14}). Let $P = \sum_{\fs \in \cI} a_{\fs} [\fs] \in \cH$ $(a_{\fs} \in k)$. For the convenience of later use, we set
    \begin{align*}
        \Li_P(\bone):=\sum_{\fs} a_{\fs} \Li_{\fs}(\bone).
    \end{align*}

\subsection{The key lemma}
The fundamental system of Frobenius difference equations that $\left\{ \cL(\bn)\right\}_{\bn\in \sI(\fs)}$ satisfy  is given in~\cite[(5.3.3),(5.3.4)]{C14} as well as~\cite[(2.3.4), (2.3.7)]{CPY19} with all $Q_{i}=1$ there, and it plays a crucial role in the proof of the following Lemma when applying ABP-criterion. 
We further mention that the first formulation of the following Lemma arises from the ideas in the proof of~\cite[Thm.~2.5.2]{CPY19}, and the second one is an extension of part of Step~3 in the proof of~\cite[Thm.~6]{ND21}, which dealt with $w\leq 2q-2$.

\begin{lemma} \label{lemma-rational}
Let $w \geq 0$ and $\emptyset\neq M \subset \cI_{w}$.
Let $P = \sum_{\fs \in \cI} a_{\fs} [\fs] \in \cH$ $(a_{\fs} \in k)$ and suppose that
\begin{itemize}
\item $\Supp(P) \subset M$,
\item $\Li_{P}(\bone)$ is Eulerian, that is, $\Li_{P}(\bone) = \sum_{\fs} a_{\fs} \Li_{\fs}(\bone) \in k \cdot \tpi^{w}$,
\item $\Li_{\fs}(\bone)$ $(\fs \in \sI(M) \cap \cI_{\ell})$ are linearly independent over $k$ for each $0 \leq \ell < w$.
\end{itemize}
Then the following hold:
\begin{enumerate}
\item
We have
\begin{align*}
\sum_{\fs' \in \cI_{w - \wt(\fs)}} a_{(\fs, \fs')} \Li_{\fs'}(\bone) = 0
\ \ \textrm{or} \ \
(q - 1) \mid (w - \wt(\fs))
\end{align*}
for each $\fs \in \sI(M)$.
In particular, if $\Li_{\fs'}(\bone)$ $(\fs' \in \cI_{w - \wt(\fs)}, (\fs, \fs') \in \Supp(P))$ are linearly independent over $k$ for each $\fs \in \sI(\Supp(P)) \setminus \{ \emptyset \}$ then
\begin{align*}
\Supp(P) \subset \{ (s_{1}, \ldots, s_{r}) \in \cI_{w} \ | \ (q - 1) \mid s_{2}, \ldots, s_{r} \}.
\end{align*}

\item
The system of Frobenius equations \eqref{eq-Frob} has a solution $(\varepsilon_{\fs})_{\fs \in \sI(M)} \in \Fq(t)[\theta]^{\sI(M)}$ such that $\varepsilon_{\fs} = a_{\fs}|_{\theta = t}$ for each $\fs \in M$.
\end{enumerate}
\end{lemma}

\begin{proof} In what follows, our essential arguments are rooted  in the ideas of the proof of~\cite{CPY19}.
Let $\alpha_{\fs} = \alpha_{\fs}(t) := a_{\fs}|_{\theta = t} \in \Fq(t)$, $c := \Li_{P}(\bone) / \tpi^{w} \in k$ and $\sI(M)' := \sI(M) \setminus M$.
We may assume that $w > 0$, $P \neq 0$ and $\alpha_{\fs} \in \Fq[t]$ $(\fs \in \cI)$. In particular, $\sI(M)' \neq \emptyset$.
We define matrices $\Phi' = (\Phi'_{\fn, \fs})_{\fn, \fs \in \sI(M)'} \in \Mat_{|\sI(M)'|}(\ok[t])$ and $\Psi' = (\Psi'_{\fn, \fs})_{\fn, \fs \in \sI(M)'} \in \GL_{|\sI(M)'|}(\TT)$ indexed by the set $\sI(M)'$. They are defined by
\begin{align*}
\Phi'_{\fn, \fs} := \left\{ \begin{array}{ll} (t - \theta)^{w - \wt(\fs)} & (\fs = \fn \ \textrm{or} \ \fs = \fn_{+}) \\ 0 & (\textrm{otherwise}) \end{array} \right.
\textrm{and} \
\Psi'_{\fn, \fs} := \left\{ \begin{array}{ll} \cL(\fs') \Omega^{w - \wt(\fn)} & (\fn = (\fs, \fs')) \\ 0 & (\textrm{otherwise}) \end{array} \right..
\end{align*}
In particular, we have $\Psi'_{\fn, \emptyset} = \cL(\fn) \Omega^{w - \wt(\fn)}$.
We also define row vectors $\bv = (v_{\fs})_{\fs \in \sI(M)'} \in \Mat_{1 \times |\sI(M)'|}(\ok[t])$ and $\bff = (f_{\fs})_{\fs \in \sI(M)'} \in \Mat_{1 \times |\sI(M)'|}(\TT)$ by
\begin{align*}
v_{\fs} := \left\{ \begin{array}{ll} \alpha_{(\fs, w - \wt(\fs))} (t - \theta)^{w - \wt(\fs)} & ((\fs, w - \wt(\fs)) \in M) \\ 0 & (\textrm{otherwise}) \end{array} \right.
\textrm{and} \
f_{\fs} := \sum_{\fs' \in \cI_{w - \wt(\fs)}} \alpha_{(\fs, \fs')} \cL(\fs').
\end{align*}
In particular, we have $f_{\emptyset} := \sum_{\fs' \in \cI_{w}} \alpha_{\fs'} \cL(\fs') = \cL(P)$.
Then we set
\begin{align*}
\Phi := \left( \begin{array}{cc} \Phi' & 0 \\ \bv & 1 \end{array} \right) \in \Mat_{|\sI(M)'| + 1}(\ok[t])
\ \ \textrm{and} \ \
\Psi := \left( \begin{array}{cc} \Psi' & 0 \\ \bff & 1 \end{array} \right) \in \GL_{|\sI(M)'| + 1}(\TT).
\end{align*}
We also set
\begin{align*}
\widetilde{\Phi} := \left( \begin{array}{cc} 1 & 0 \\ 0 & \Phi \end{array} \right) \in \Mat_{|\sI(M)'| + 2}(\ok[t])
\ \ \textrm{and} \ \
\widetilde{\psi} := \left( \begin{array}{c} 1 \\ (\Psi'_{\fn, \emptyset})_{\fn \in \sI(M)'} \\ f_{\emptyset} \end{array} \right) \in \Mat_{(|\sI(M)'| + 2) \times 1}(\TT).
\end{align*}
Then using \cite[(2.3.4), (2.3.7)]{CPY19} one checks that
\begin{align*}
\Psi^{(- 1)} = \Phi \Psi
\ \ \textrm{and} \ \
\widetilde{\psi}^{(- 1)} = \widetilde{\Phi} \widetilde{\psi}.
\end{align*}
We mention that we follow~\cite{ND21} to use double indices to indicate the entries of $\Phi$ here,
and putting such $f_{\phi}$ into a system of Frobenius difference equations was first used by the first named author in~\cite{C16}, and later used in \cite{CH21, ND21}. Such $f_{\phi}$ naturally appears in the period matrix of the fiber coproduct of rigid analytically trivial dual $t$-motives in \cite{CM21}.

By \cite[Theorem 3.1.1]{ABP04}, there exist $g \in \ok(t)$ and $\bg = (g_{\fs})_{\fs \in \sI(M)'} \in \Mat_{1 \times |\sI(M)'|}(\ok(t))$ such that
\begin{align*}
\widetilde{\bg} \widetilde{\psi} = 0, \ \ \textrm{$g$ and $\bg$ are regular at $t = \theta$} \ \ \textrm{and} \ \ \widetilde{\bg}|_{t = \theta} = (- c, 0, \ldots, 0, 1),
\end{align*}
where $\widetilde{\bg} := (g, \bg, 1)$.

We set $B = B(t)$ and $B_{\fs} = B_{\fs}(t)$ in $\ok(t)$ by
\begin{align*}
(B, (B_{\fs})_{\fs \in \sI(M)'}, 0) := \widetilde{\bg} - \widetilde{\bg}^{(- 1)} \widetilde{\Phi} \in \Mat_{1 \times (|\sI(M)'| + 2)}(\ok(t)).
\end{align*}
We claim that $\widetilde{\bg}^{(- 1)} \widetilde{\Phi} = \widetilde{\bg}$, that is, $B = B_{\fs} = 0$ for all $\fs \in \sI(M)'$.
Indeed,
\begin{align*}
B + \sum_{\fs \in \sI(M)'} B_{\fs} \cL(\fs) \Omega^{w - \wt(\fs)}
= (B, (B_{\fs})_{\fs \in \sI(M)'}, 0) \widetilde{\psi}
= (\widetilde{\bg} - \widetilde{\bg}^{(- 1)} \widetilde{\Phi}) \widetilde{\psi}
= \widetilde{\bg} \widetilde{\psi} - (\widetilde{\bg} \widetilde{\psi})^{(- 1)}
= 0
\end{align*}
and $\sI(M)' = \bigsqcup_{\ell = 0}^{w - 1} \sI(M) \cap \cI_{\ell}$ imply the equality
\begin{align} \label{eq-B}
B + \sum_{\fs \in \sI(M) \cap \cI_{w - 1}} B_{\fs} \cL(\fs) \Omega + \sum_{\fs \in \sI(M) \cap \cI_{w - 2}} B_{\fs} \cL(\fs) \Omega^{2} + \cdots + \sum_{\fs \in \sI(M) \cap \cI_{0}} B_{\fs} \cL(\fs) \Omega^{w} = 0.
\end{align}
We note that
\begin{itemize}
\item $B$ and $B_{\fs}$ $(\fs \in \sI(M)')$ are regular at $t = \theta^{q^{i}}$ for $j \gg 0$,
\item $\Omega$ and $\cL(\fs)$ $(\fs \in \cI)$ are entire,
\item $\Omega(\theta^{q^{j}}) = 0$ for $j \geq 1$,
\item $\cL(\fs)|_{t = \theta^{q^{j}}} = (\Li_{\fs}(\bone) / \tpi^{\wt(\fs)})^{q^{j}}$ for $\fs \in \cI$ and $j \geq 0$.
\end{itemize}
Thus evaluating at $t = \theta^{q^{j}}$ for $j \gg 0$ in \eqref{eq-B} implies $B(\theta^{q^{j}}) = 0$.
Since $B$ is rational, we have $B = 0$.
Then dividing \eqref{eq-B} by $\Omega$ and evaluating at $t = \theta^{q^{j}}$ for $j \gg 0$ imply
\begin{align*}
\sum_{\fs \in \sI(M) \cap \cI_{w - 1}} B_{\fs}(\theta^{q^{j}}) (\Li_{\fs}(\bone) / \tpi^{w - 1})^{q^{j}} = 0.
\end{align*}
This is equivalent to
\begin{align*}
\sum_{\fs \in \sI(M) \cap \cI_{w - 1}} B_{\fs}(\theta^{q^{j}})^{q^{- j}} \Li_{\fs}(\bone) = 0.
\end{align*}
By the assumption of linear independence and \cite[Theorem 2.2.1]{C14}, we have $B_{\fs}(\theta^{q^{j}}) = 0$.
Thus we have $B_{\fs} = 0$ for all $\fs \in \sI(M) \cap \cI_{w - 1}$.
Repeating this process, we have $B_{\fs} = 0$ for all $\fs \in \sI(M)'$.
We note that the claim implies
\begin{align} \label{eq-lemma-claim}
\left( \begin{array}{cc} \Id & 0 \\ \bg & 1 \end{array} \right)^{(- 1)} \left( \begin{array}{cc} \Phi' & 0 \\ \bv & 1 \end{array} \right) = \left( \begin{array}{cc} \Phi' & 0 \\ 0 & 1 \end{array} \right) \left( \begin{array}{cc} \Id & 0 \\ \bg & 1 \end{array} \right).
\end{align}

(1)
We set
\begin{align*}
X := \left( \begin{array}{cc} \Id & 0 \\ \bg & 1 \end{array} \right) \Psi = \left( \begin{array}{cc} \Psi' & 0 \\ \bg \Psi' + \bff & 1 \end{array} \right) \in \GL_{|\sI(M)'| + 1}(\Frac(\TT)).
\end{align*}
By using \eqref{eq-lemma-claim}, we can verify that $X^{(- 1)} = \left( \begin{array}{cc} \Phi' & 0 \\ 0 & 1 \end{array} \right) X$.
Thus by \cite[\S 4.1.6]{P08},
\begin{align*}
\GL_{|\sI(M)'| + 1}(\Fq(t))
\ni \left( \begin{array}{cc} \Psi' & 0 \\ 0 & 1 \end{array} \right)^{- 1} X
= \left( \begin{array}{cc} \Id & 0 \\ \bg \Psi' + \bff & 1 \end{array} \right)
=: \left( \begin{array}{cc} \Id & 0 \\ (h_{\fs})_{\fs \in \sI(M)'} & 1 \end{array} \right).
\end{align*}
Then evaluating at $t = \theta^{q^{j}}$ for $j \gg 0$ in $\bg \Psi' + \bff = (h_{\fs})_{\fs \in \sI(M)'}$, we have $f_{\fs}(\theta^{q^{j}}) = h_{\fs}(\theta^{q^{j}}) = h_{\fs}(\theta)^{q^{j}}$.
Since
\begin{align*}
f_{\fs}(\theta^{q^{j}}) = \sum_{\fs' \in \cI_{w - \wt(\fs)}} \alpha_{(\fs, \fs')}(\theta^{q^{j}}) \cL(\fs')|_{t = \theta^{q^{j}}}
= \sum_{\fs' \in \cI_{w - \wt(\fs)}} \alpha_{(\fs, \fs')}(\theta)^{q^{j}} (\Li_{\fs'}(\bone) / \tpi^{w - \wt(\fs)})^{q^{j}}
= f_{\fs}(\theta)^{q^{j}},
\end{align*}
we have $f_{\fs}(\theta) = h_{\fs}(\theta) \in k$ for each $\fs \in \sI(M)'$.
Thus we have
\begin{align*}
\sum_{\fs' \in \cI_{w - \wt(\fs)}} a_{(\fs, \fs')} \Li_{\fs'}(\bone) \in k \cdot \tpi^{w - \wt(\fs)}
\end{align*}
for each $\fs \in \sI(M)'$.
We note that this holds whenever $\fs \in M$.
Then the first statement of (1) follows from this relation because $\tpi \in (- \theta)^{\frac{1}{q - 1}} \cdot k_{\infty}^{\times}$.

Next we prove the second statement of (1).
Let $(\fs, \fs'') \in \Supp(P)$ with $\fs \neq \emptyset$.
Then by the first statement of (1) and the assumption on the linear independence of the second statement of (1), we have $(q - 1) \mid (w - \wt(\fs)) = \wt(\fs'')$.

(2)
By \eqref{eq-lemma-claim} and \cite[Proposition 2.2.1]{CPY19}, there exists $\alpha \in \Fq[t] \setminus \{ 0 \}$ such that $\alpha \bg \in \Mat_{1 \times |\sI(M)|}(\ok[t])$.
If we set $\varepsilon_{\fs} := g_{\fs}^{(- 1)}$ $(\fs \in \sI(M)')$ and $\varepsilon_{\fs} := \alpha_{\fs} \in \Fq(t)$ $(\fs \in M)$, then
\begin{align*}
((\varepsilon_{\fs})_{\fs \in \sI(M)'}, 1) \left( \begin{array}{c} \Phi' \\ \bv \end{array} \right) = (\varepsilon_{\fs}^{(1)})_{\fs \in \sI(M)'}
\ \ \textrm{and} \ \
\varepsilon_{\fs}^{(1)} = \varepsilon_{\fs} \ (\fs \in M)
\end{align*}
give the desired equations.
%By a similar argument of the proof of \cite[Theorem 2]{KL16}, we have $\delta_{\fs} \in \Fq(t)[\theta]$ and $\deg_{\theta} \delta_{\fs} \leq (1 + 1 / (q - 1)) (w - \wt(\fs))$.
\end{proof}

\subsection{Linear independence}
With the crucial properties of Lemma~\ref{lemma-rational} established, we are able to show the following linear independence result.

\begin{theorem} \label{theorem-IND-basis}
For each $w \geq 0$, the elements
\begin{align*}
\Li_{\fs}(\bone) \ \ \textrm{with} \ \ \fs \in \INDw
\end{align*}
form a basis of $\cZ_{w}$.
\end{theorem}

\begin{proof} (cf.~\cite[p.~388]{ND21} for the special case of $k$-linear independence of  $\zeta_{A}(w)$ and $\zeta_{A}(w - (q - 1), q - 1)$ with $w \leq 2 q - 2$.)
By Theorem~\ref{theorem-generator},  $\left\{ \Li_{\fs}(\bone)| \fs\in \INDw  \right\}$ is a generating set for $\cZ_{w}$.
Thus it is enough to show that the given elements are linearly independent over $k$.
This is proved by induction on $w$.
When $w = 0$ we have $\IND_{0} = \{ \emptyset \}$, and $\Li_{\emptyset}(\bone) = 1$ is linearly independent over $k$.

Let $w \geq 1$ and assume that the elements $\Li_{\fs}(\bone)$ with $\fs \in \IND_{w'}$ are linearly independent over $k$ for $w' < w$.
Let $P = \sum_{\fs \in \cI} \alpha_{\fs}(\theta) [\fs] \in \cH$ $(\alpha_{\fs} = \alpha_{\fs}(t) \in \Fq(t))$ be a linear relation among $\Li_{\fs}(\bone)$'s over $k$ such that $\Supp(P) \subset \INDw$.
By Lemma \ref{lemma-rational} (1) for $M = \Supp(P)$ and the induction hypothesis, we have $\Supp(P) \subset \INDzw$.
By Lemma \ref{lemma-rational} (2) for $M = \INDzw$, there exists a solution $(\varepsilon_{\fs}) \in \sX_{w}$ such that $\varepsilon_{\fs} = \alpha_{\fs}$ for all $\fs \in \INDzw$.

When $(q - 1) \nmid w$, Theorem \ref{theorem-dimension} implies that $\alpha_{\fs} = 0$ for all $\fs \in \INDzw$.
When $(q - 1) \mid w$, Theorem \ref{theorem-generator} implies that there exists $P_{w} = \sum_{\fs \in \cI} \beta_{\fs}(\theta) [\fs] \in \cH_{w}$ $(\beta_{\fs}(t) \in \Fq(t))$ such that $\Supp(P_{w}) \subset \INDw$ and $\Li_{P_{w}}(\bone) = \tpi^{w}$. 
It is clear that $P_{w} \neq 0$.
By Lemma \ref{lemma-rational} (1) for $M = \Supp(P_{w})$, we have $\Supp(P_{w}) \subset \INDzw$.
By Lemma \ref{lemma-rational} (2) for $M = \INDzw$, there exists a solution $(\varepsilon_{\fs}') \in \sX_{w}$ such that $\varepsilon'_{\fs} = \beta_{\fs}$ for all $\fs \in \INDzw$.
Then by Theorem~\ref{theorem-dimension}, there exists an $\alpha \in \FF_{q}(t)$ for which $(\varepsilon_{\fs})=\alpha (\varepsilon_{\fs}')$, and hence we have $P = \alpha(\theta) P_{w}$.
Then we have
\begin{align*}
0 = \Li_{P}(\bone) = \Li_{\alpha (\theta) P_{w}}(\bone) = \alpha(\theta) \tpi^{w}.
\end{align*}
Thus $\alpha = 0$ and $P = 0$.
\end{proof}

\subsection{Proof of Theorem~\ref{T:Main Thm}}\label{Sec: Proof of Main Thm} With the fundamental results established, we can now give a short proof of Theorem~\ref{T:Main Thm}. Given $w\in \ZZ_{>0}$, we claim that $\mathcal{B}_{w}^{\rm{T}}$ is a basis of the $k$-vector space $\cZ_{w}$. By Theorem~\ref{theorem-IND-basis}, $\left\{ \Li_{\fs}(\bone)| \fs\in \INDw  \right\}$ is a $k$-basis of $\cZ_{w}$. Since $|\INDw|=|\ITw|$ by Proposition~\ref{Pop:|ITw|=|INDw|}, and $\cB^{\rm{T}}_{w}$ is a generating set of $\cZ_{w}$ by Corollary \ref{Cor:generating set},
we have that
\[|\mathcal{B}_{w}^{\rm{T}}| \geq \dim_{k} \cZ_{w}=|\INDw|=|\ITw|\geq |\mathcal{B}_{w}^{\rm{T}}|,\] whence the desired result follows.

\subsection{Generating set of linear relations}\label{Sub:Generating set of linear relations}
Recall the $k$-linear map $\sU^{\zeta}$ defined in Definition~\ref{def-sU}, and Theorem~\ref{theorem-algo} asserts that for $\fs\in \cI_{w} \setminus \ITw$,
\begin{equation}\label{E:Main Relations}
\sL^{\zeta}\left( [\fs]-\sU^{\zeta}(\fs) \right)=0. 
\end{equation}
We prove in the following theorem, which verifies the $\sB^{*}$-version of~\cite[Conjecture~5.1]{To18}, that these relations account for all $k$-linear relations among  MZV's of the same weight. 

\begin{theorem}\label{T: DetermineLR}
Let $w$ be a positive integer.  Then all the $k$-linear relations among the MZV's of weight $w$ are generated by~\eqref{E:Main Relations} for $\fs\in \cI_{w} \setminus  \ITw$.  In other words, if we denote by $\sLZ_{w}:= \sLZ|_{\cH_{w}}:=\left( [\fs]\mapsto \zeta_{A}(\fs) \right) :\cH_{w} \twoheadrightarrow \cZ_{w}$ and put
\[ \sR_{w}:=\Span_{k}\left\{[\fs]-\sU^{\zeta}(\fs) \mid  \fs\in \cI_{w} \setminus  \ITw \right\}\subset \cH_{w} ,\] 
 then 
\[\Ker\sLZ_{w}= \sR_{w}.\]
Moreover, we have \[\Ker \sLZ=\bigoplus_{w\in \NN}\sR_{w}.\]
\end{theorem}

\begin{proof}
Theorem~\ref{theorem-algo} implies the inclusion $\Ker \sLZ_{w} \supset \sR_{w}$, and Remark~\ref{Rem:Us=s} implies
\[ \sR_{w}=\Span_{k}\left\{P-\sU^{\zeta}(P) \mid  P \in \cH_{w} \right\}.\]
We fix a positive integer $e$ given in Theorem~\ref{theorem-algo}.
Then for each $P \in \cH_{w}$, we have
\[ P \equiv \sU^{\zeta}(P) \equiv \cdots \equiv (\sU^{\zeta})^{e}(P) \bmod \sR_{w}.\]
By these relations and Theorem~\ref{theorem-algo}, the quotient space $\cH_{w} / \sR_{w}$ is spanned by the image of $\ITw$.
It follows that
\[ |\IT_{w}| \geq \dim_{k} \cH_{w} / \sR_{w} \geq \dim_{k} \cH_{w} / \Ker \sLZ_{w} = \dim_{k} \cZ_{w} = |\IT_{w}|, \]
where the last equality comes from Theorem~\ref{T:Main Thm}.
So we have $\Ker \sLZ_{w}=\sR_{w}$.  By~\cite[Thm.~2.2.1]{C14},  the last assertion follows.
\end{proof}

\appendix

\section{Proof of Theorem \ref{theorem-algo}}

The aim of this appendix is to give a detailed proof of Theorem \ref{theorem-algo}.

\subsection{Inequalities of depth}

\begin{proposition} \label{prop-product-depth}
Let $\bullet\in \left\{\Li,\zeta \right\}$ and $\fs, \fn \in \cI$ be indices. Then for each $\fu \in \Supp(\fs *^{\bullet} \fn)$, we have
\begin{align*}
\max\{ \dep(\fs), \dep(\fn) \} \leq \dep(\fu) \leq \dep(\fs) + \dep(\fn).
\end{align*}

\end{proposition}

\begin{proof}
Put $r := \dep(\fs)$ and $\ell := \dep(\fn)$. We prove the inequalities by induction on $r + \ell$. Note that in the case of $r = 0$ or $\ell = 0$, namely, $\fs=\emptyset $ or $\fn=\emptyset$, the result follows from the definition of $*^{\bullet}$.

Suppose that $r\geq 1$ and  $\ell \geq 1$. Since the binary operation $*^{\bullet}$ is commutative,  without loss of generality we may assume that $r \geq \ell \geq 1$.
We first consider the case that $r > \ell$.
By the induction hypothesis, we have:
\begin{itemize}
\item
when $\fu \in \Supp([s_{1}, \fs_{-} *^{\bullet} \fn])$, then $1 + (r - 1) \leq \dep(\fu) \leq 1 + (r - 1) + \ell$;

\item
when $\fu \in \Supp([n_{1}, \fs *^{\bullet} \fn_{-}])$, then $1 + r \leq \dep(\fu) \leq 1 + r + (\ell - 1)$;

\item
when $\fu \in \Supp([s_{1} + n_{1}, \fs_{-} *^{\bullet} \fn_{-}])$, then $1 + (r - 1) \leq \dep(\fu) \leq 1 + (r - 1) + (\ell - 1)$;

\item
when $\fu \in \Supp(D^{\zeta}_{\fs, \fn})$, then there exist $j$, $\fu'$ and $\fu''$ such that $1 \leq j < s_{1} + n_{1}$, $\fu'' \in \Supp(\fs_{-} *^{\zeta} \fn_{-})$ and $\fu' \in \Supp((j) *^{\zeta} \fu'')$ with $\fu = (s_{1} + n_{1} - j, \fu')$. It follows that $r - 1 \leq \dep(\fu'') \leq (r - 1) + (\ell - 1)$, and hence $r - 1 \leq \dep(\fu') \leq 1 + (r - 1) + (\ell - 1)$.
\end{itemize}
In any case, we have $r \leq \dep(\fu) \leq r + \ell$.
In the case of $r = \ell$, we can verify that $r \leq \dep(\fu) \leq 2 r$ by a similar argument as above.
\end{proof}

\begin{corollary} \label{cor-depth}
Let $\fs, \fn \in \cI$ be indices.
\begin{enumerate}
\item
Suppose that both $\fs$ and $\fn$ are non-empty indices. Let $D^{\zeta}_{\fs, \fn}$ be defined in~\eqref{E:DZetaS}. Then for each $\fu \in \Supp(D^{\zeta}_{\fs, \fn})$, we have
\begin{align*}
\max\{ \dep(\fs), \dep(\fn) \} \leq \dep(\fu) \leq \dep(\fs) + \dep(\fn).
\end{align*}

\item
For any integer $m \geq 0$, we let $\alpha^{\bullet,m}_{q}$ be the $m$-th iteration of $\alpha^{\bullet}_{q}$ defined in~\eqref{E:alpha}. Then for each $\fu \in \Supp(\alpha^{\bullet, m}_{q}(\fs))$, we have
\begin{align*}
\dep(\fu) \geq
\left\{ \begin{array}{@{}ll}
1 + m & (\fs = \emptyset, \ m \geq 1) \\
\dep(\fs) + m & (\textrm{otherwise})
\end{array} \right.
\ \ \textrm{and} \ \
\dep(\fu) \leq \dep(\fs) + 2 m.
\end{align*}

\end{enumerate}
\end{corollary}

\begin{proof}
This is a direct consequence of Proposition \ref{prop-product-depth}.
\end{proof}

\subsection{Binary relations}

Let $\bullet \in \{ \Li, \zeta \}$.
For each $d \in \ZZ$ and $R = (P, Q) \in \cH^{\oplus 2}$, we define
\begin{align*}
\sLB(R) := \sLB(P) + \sLB(Q)
\ \ \textrm{and} \ \
\sLB_{d}(R) := \sLB_{d}(P) + \sLB_{d + 1}(Q).
\end{align*}
We set
\begin{align}\label{E:Pbullet}
\cP^{\bullet} &:= \{ (P, Q)\in \cH^{\oplus 2} \ | \ \sLB_{d}(P) + \sLB_{d + 1}(Q) = 0 \ \textrm{for all} \ d \in \ZZ \}
\end{align}
and
\begin{equation}\label{E:cPw}
\cP^{\bullet}_{w} := \cP^{\bullet} \cap \cH_{w}^{\oplus 2}
\end{equation}
for $w \geq 0$.
Elements in $\cP^{\bullet}$ are called \textit{binary relations}.
We note that each binary relation $R = (P, Q) \in \cP^{\bullet}$ induces a $k$-linear relation $\sLB(R) = \sLB(P) + \sLB(Q) = 0$.
Indeed, we have
\begin{align*}
\sLB(P) + \sLB(Q) = \sum_{d \in \ZZ} \sLB_{d}(P) + \sum_{d \in \ZZ} \sLB_{d + 1}(Q) = \sum_{d \in \ZZ} (\sLB_{d}(P) + \sLB_{d + 1}(Q)) = 0.
\end{align*} The above ideas were introduced by Todd~\cite{To18}.
\begin{example}
Thakur proved the identity \cite[Thm.~5]{T09b}
\begin{align*}
\sLZ_{d}(q) - L_{1} \sLZ_{d + 1}(1, q - 1) &= 0 \ \ \ (d \in \ZZ).
\end{align*}
Note that by Remark \ref{rmk-C=A} we can replace `$\zeta$' by `$\Li$' in the above equation. It follows that  for $\bullet\in \left\{\Li, \zeta \right\}$, we have a binary relation
\begin{equation}\label{E:R1}
R_{1} := ([q], \ - L_{1} [1, q - 1]) \in \cP^{\bullet}_{q};
\end{equation}
namely for each $d\in \ZZ$, the following equation holds:
\[ \sLB_{d}(q) - L_{1} \sLB_{d + 1}(1, q - 1) = 0.\]
\end{example}

\subsection{Maps between relations}\label{Sec:Maps B,C,BC}
In what follows, the essential ideas we use are rooted in~\cite{To18, ND21}. Let $\fs = (s_{1}, \ldots, s_{r}) \in \cI_{> 0}$ be a non-empty index, we recall that $\fs_{+} := (s_{1}, \ldots, s_{r - 1})$ and $\fs_{-} := (s_{2}, \ldots, s_{r})$, and the operator $\boxplus: \cH^{\oplus 2}\rightarrow \cH$ is defined in Definition~\ref{Def:boxplus}. For $\bullet\in \left\{\Li,\zeta \right\}$, we define the maps $\sB^{\bullet}_{\fs}, \sC^{\bullet}_{\fs}:\cH_{> 0}^{\oplus 2}\rightarrow \cH_{> 0}^{\oplus 2} $ as follows: For each $R = (P, Q) = (\sum_{\fn \in \cI_{> 0}} a_{\fn} [\fn], \ \sum_{\fn \in \cI_{> 0}} b_{\fn} [\fn]) \in \cH_{> 0}^{\oplus 2}$, we set
\begin{align*}
\sB^{\bullet}_{\fs}(R) &:= \left( [\fs, P] + [\fs, Q] + (\fs \boxplus Q) + [\fs_{+}, D^{\bullet}_{s_{r}, Q}], \ 0 \right), \\
\sC^{\bullet}_{\fs}(R) &:= \sum_{\fn = (n_{1}, \fn_{-}) \in \cI_{> 0}} \left( a_{\fn} [n_{1} + s_{1}, \fn_{-} *^{\bullet} \fs_{-}] + a_{\fn} [n_{1}, \fn_{-} *^{\bullet} \fs] + a_{\fn} D^{\bullet}_{\fn, \fs}, \ b_{\fn} [n_{1}, \fn_{-} *^{\bullet} \fs] \right), 
\end{align*}
where  \[D^{\bullet}_{s, Q} := \sum_{\fn \in \cI_{> 0}} b_{\fn} D^{\bullet}_{s, \fn}.\] For any integer $m\geq 0$, we further define the map $\sBC^{\bullet, m}_{q}: \cH_{> 0}^{\oplus 2}\rightarrow \cH_{> 0}^{\oplus 2}$ given by:
\[
\sBC^{\bullet, m}_{q}(R) := \left( [q^{\{m\}}, P], \ L_{1}^{m} \alpha^{\bullet, m}_{q}(Q) \right).
\]
It is clear that these maps are $k$-linear endomorphisms on $\cH_{> 0}^{\oplus 2}$.

\begin{proposition} \label{prop-BC}
Let $\bullet\in \left\{\Li,\zeta \right\}$.
For each  $\fs \in \cI_{> 0}$, and integers $m, w$ satisfying $m \geq 0, w > 0$, the maps $\sB^{\bullet}_{\fs}$, $\sC^{\bullet}_{\fs}$ and $\sBC^{\bullet, m}_{q}$ satisfy
\begin{align*}
\sB^{\bullet}_{\fs}(\cP^{\bullet}_{w}) \subset \cP^{\bullet}_{w + \wt(\fs)},
\ \ \
\sC^{\bullet}_{\fs}(\cP^{\bullet}_{w}) \subset \cP^{\bullet}_{w + \wt(\fs)}
\ \ \ \textrm{and} \ \ \
\sBC^{\bullet, m}_{q}(\cP^{\bullet}_{w}) \subset \cP^{\bullet}_{w + m q},
\end{align*}
where $\cP^{\bullet}_w$ is defined in~\eqref{E:cPw}.

\end{proposition}

\begin{proof}
When $R \in \cP^{\bullet}_{w}$, it is clear that
\begin{align*}
\sB^{\bullet}_{\fs}(R) \in \cH_{w + \wt(\fs)}^{\oplus 2}, \ \ \sC^{\bullet}_{\fs}(R) \in \cH_{w + \wt(\fn)}^{\oplus 2} \ \ \textrm{and} \ \ \sBC^{\bullet}_{\fs}(R) \in \cH_{w + \wt(\fs)}^{\oplus 2}.
\end{align*}
Thus it suffices to show that $\sB^{\bullet}_{\fs}(R), \sC^{\bullet}_{\fs}(R), \sBC^{\bullet, m}_{q}(R) \in \cP^{\bullet}$.

For each $R = (\sum_{\fn \in \cI_{w}} a_{\fn} [\fn], \ \sum_{\fn \in \cI_{w}} b_{\fn} [\fn]) \in \cP^{\bullet}_{w}$, the corresponding equalities
\begin{align*}
\sum_{\fn \in \cI_{w}} a_{\fn} \sLB_{i}(\fn) + \sum_{\fn \in \cI_{w}} b_{\fn} \sLB_{i + 1}(\fn) = 0 \ \ \ (i \in \ZZ)
\end{align*}
hold. Thus for each $s \geq 1$ and $d \in \ZZ$, we have
\begin{align*}
0 &= \sLB_{d}(s) \sum_{i < d} \left( \sum_{\fn \in \cI_{w}} a_{\fn} \sLB_{i}(\fn) + \sum_{\fn \in \cI_{w}} b_{\fn} \sLB_{i + 1}(\fn) \right) \\
&= \sum_{\fn \in \cI_{w}} a_{\fn} \sLB_{d}(s) \sLB_{< d}(\fn) + \sum_{\fn \in \cI_{w}} b_{\fn} \sLB_{d}(s) \sLB_{< d}(\fn) + \sum_{\fn \in \cI_{w}} b_{\fn} \sLB_{d}(s) \sLB_{d}(\fn) \\
&= \sum_{\fn \in \cI_{w}} a_{\fn} \sLB_{d}(s, \fn) + \sum_{\fn \in \cI_{w}} b_{\fn} \sLB_{d}(s, \fn) + \sum_{\fn = (n_{1}, \fn_{-}) \in \cI_{w}} b_{\fn} \sLB_{d}(s + n_{1}, \fn_{-}) + \sum_{\fn \in \cI_{w}} b_{\fn} \sLB_{d}(D^{\bullet}_{s, \fn}) \\
&= \sLB_{d}(\sB^{\bullet}_{s}(R)),
\end{align*}
where the third equality comes from Proposition~\ref{P:product of sLbullet}. This means that $\sB^{\bullet}_{s}(R) \in \cP^{\bullet}$. Note that for each $\fs = (s_{1}, \ldots, s_{r}) \in \cI_{> 0}$, we have $\sB^{\bullet}_{\fs}(R) \in \cP^{\bullet}$ since from the definition of $\sB^{\bullet}_{\fs}$ we have
\begin{align*}
\sB^{\bullet}_{\fs} = \sB^{\bullet}_{s_{1}} \circ \cdots \circ \sB^{\bullet}_{s_{r}}.
\end{align*}

Similarly, we have
\begin{align*}
0 &= \left( \sum_{\fn \in \cI_{w}} a_{\fn} \sLB_{d}(\fn) + \sum_{\fn \in \cI_{w}} b_{\fn} \sLB_{d + 1}(\fn) \right) \sLB_{< d + 1}(\fs) \\
&= \sum_{\fn \in \cI_{w}} a_{\fn} \sLB_{d}(\fn) \sLB_{d}(\fs) + \sum_{\fn = (n_{1}, \fn_{-}) \in \cI_{w}} a_{\fn} \sLB_{d}(n_{1}) \sLB_{< d}(\fn_{-}) \sLB_{< d}(\fs) \\
& \ \ \ + \sum_{\fn = (n_{1}, \fn_{-}) \in \cI_{w}} b_{\fn} \sLB_{d + 1}(n_{1}) \sLB_{< d + 1}(\fn_{-}) \sLB_{< d + 1}(\fs) \\
& = \sum_{\fn = (n_{1}, \fn_{-}) \in \cI_{w}} a_{\fn} \left( \sLB_{d}([n_{1} + s_{1}, \fn_{-} *^{\bullet} \fs_{-}]) + \sLB_{d}(D^{\bullet}_{\fn, \fs}) \right) + \sum_{\fn = (n_{1}, \fn_{-}) \in \cI_{w}} a_{\fn} \sLB_{d}([n_{1}, \fn_{-} *^{\bullet} \fs]) \\
& \ \ \ + \sum_{\fn = (n_{1}, \fn_{-}) \in \cI_{w}} b_{\fn} \sLB_{d + 1}([n_{1}, \fn_{-} *^{\bullet} \fs]) \\
& = \sLB_{d}(\sC^{\bullet}_{\fn}(R)),
\end{align*}
where the third equality comes from Proposition~\ref{P:product of sLbullet}. Thus, we have $\sC^{\bullet}_{\fs}(R) \in \cP^{\bullet}$.

Finally, we have
\begin{align*}
\cP^{\bullet} &\ni \sB^{\bullet}_{q}(R) - \sum_{\fn \in \cI_{w}} b_{\fn} \sC^{\bullet}_{\fn}(R_{1}) \\
&= \sum_{\fn = (n_{1}, \fn_{-}) \in \cI_{w}} \left( a_{\fn} [q, \fn] + b_{\fn} [q, \fn] + b_{\fn} [q + n_{1}, \fn_{-}] + b_{\fn} D^{\bullet}_{q, \fn}, \ 0 \right) \\
& \ \ \ \ \ - \sum_{\fn = (n_{1}, \fn_{-}) \in \cI_{w}} b_{\fn} \left( \left( [q + n_{1}, \fn_{-}] + [q, \fn] + D^{\bullet}_{q, \fn} \right), \ - L_{1} [1, (q - 1) *^{\bullet} \fn] \right) \\
&= \sum_{\fn \in \cI_{w}} \left( a_{\fn} [q, \fn], \ b_{\fn} L_{1} \alpha^{\bullet}_{q}(\fn) \right) \\
&= \sBC^{\bullet}_{q}(R),
\end{align*}
where the second equality comes from the definition of $\alpha^{\bullet}_{q}$ given in~\eqref{E:alpha}. Since
\begin{align*}
\sBC^{\bullet, m}_{q} = \sBC^{\bullet}_{q} \circ \cdots \circ \sBC^{\bullet}_{q} \ \ (\textrm{$m$th iterate of $\sBC^{\bullet}_{q}$}),
\end{align*}
it shows that $\sBC^{\bullet, m}_{q}(R) \in \cP^{\bullet}$.
\end{proof}

Let $\fs \in \cI$.
We write $\fs = (s_{1}, \ldots ) = (\fs^{\rT}, q^{\{ m \}}, \fs')$ with $\fs^{\rT} \in \IT$, $m \geq 0$ and $\fs' = (s'_{1}, \ldots)$ with $s'_{1} > q$ or $\fs' = \emptyset$, and set $\Init(\fs) := (\fs^{\rT}, q^{\{ m \}})$ and $\ell_{1} := \dep(\Init(\fs))$.
When $\fs' \neq \emptyset$, we set \[\fs'' := (s'_{1} - q, \fs'_{-}) = (s_{\ell_{1} + 1} - q, \fs'_{-}).\]
We recall $\sUB(\fs)$ given in Definition \ref{def-sU} by
\begin{align*}
\sUB(\fs) := \left\{ \begin{array}{@{}ll}
- [\fs^{\rT}, q^{\{m + 1\}}, \fs''] + L_{1}^{m + 1} [\fs^{\rT}, \alpha_{q}^{\bullet, m + 1}(\fs'')] & \\
\hspace{5.0em} + L_{1}^{m + 1} (\fs^{\rT} \boxplus \alpha_{q}^{\bullet, m + 1}(\fs'')) - [\fs^{\rT}, q^{\{m\}}, D^{\bullet}_{q, \fs''}] & (\fs' \neq \emptyset) \\[0.5em]
L_{1}^{m} [\fs^{\rT}, \alpha_{q}^{\bullet, m}(\emptyset)] + L_{1}^{m} (\fs^{\rT} \boxplus \alpha_{q}^{\bullet, m}(\emptyset)) & (\fs' = \emptyset)
\end{array} \right.
\end{align*}
It is clear that $\sUB(\fs) \in \cH_{\wt(\fs)}$.

\begin{theorem}[Theorem \ref{theorem-algo} (1)] \label{theorem-sLsU}
For each $P \in \cH$, we have $\sLB(\sUB(P)) = \sLB(P)$.
\end{theorem}

\begin{proof}
We may assume that $P = \fs \in \cI$ as above.
By the definition of $\sU^{\bullet}$, it suffices to show the desired result in the case $\fs \notin \IT$.
We first consider the case that $\fs' \neq \emptyset$. Recall that $R_1$ is given in~\eqref{E:R1} and $\alpha^{\bullet}_{q}$ is defined in~\eqref{E:alpha}. Note that by definition, we have
\begin{align*}
\sBC^{\bullet, m}_{q}\left( (\sC^{\bullet}_{\fs''}(R_{1}))\right)
&= \sBC^{\bullet, m}_{q} \left( [\fs'] + [q, \fs''] + D^{\bullet}_{q, \fs''}, \ - L_{1} \alpha^{\bullet}_{q}(\fs'') \right) \\
&= \left( [q^{\{m\}}, \fs'] + [q^{\{m + 1\}}, \fs''] + [q^{\{m\}}, D^{\bullet}_{q, \fs''}], \ - L_{1}^{m + 1} \alpha^{\bullet, m + 1}_{q}(\fs'') \right).
\end{align*}
If $\fs^{\rT} = \emptyset$, we can express
\begin{equation}\label{E:Appendix EQ1}
\sBC^{\bullet, m}_{q}(\sC^{\bullet}_{\fs''}(R_{1}))
= \left( [\fs] + [\fs^{\rT}, q^{\{m + 1\}}, \fs''] + [\fs^{\rT}, q^{\{m\}}, D^{\bullet}_{q, \fs''}], \ - L_{1}^{m + 1} [\fs^{\rT}, \alpha^{\bullet, m + 1}_{q}(\fs'')] \right).
\end{equation}
If $\fs^{\rT} \neq \emptyset$,   we have

\begin{align}\label{E:Appendix EQ2}
&\sB^{\bullet}_{\fs^{\rT}}(\sBC^{\bullet, m}_{q}(\sC^{\bullet}_{\fs''}(R_{1}))) \\
&= \left( [\fs] + [\fs^{\rT}, q^{\{m + 1\}}, \fs''] + [\fs^{\rT}, q^{\{m\}}, D^{\bullet}_{q, \fs''}] - L_{1}^{m + 1} [\fs^{\rT}, \alpha^{\bullet, m + 1}_{q}(\fs'')] - L_{1}^{m + 1} [\fs^{\rT} \boxplus \alpha^{\bullet, m + 1}_{q}(\fs'')], \ 0 \right), \nonumber
\end{align}
where we use Remark~\ref{remark-D-vanish} and the fact $[\fs^{\rT}_{+},0]=0\in \cH$ from Remark~\ref{Rem: [s,0]=0}.

Next we suppose that $\fs' = \emptyset$.
Then $m \geq 1$ and we have
\begin{align*}
\sBC^{\bullet, m - 1}_{q}(R_{1})
= \left( [q^{\{m\}}], \ L_{1}^{m - 1} \alpha^{\bullet, m - 1}_{q}(- L_{1} [1, q - 1]) \right)
= \left( [q^{\{m\}}], \ - L_{1}^{m} \alpha^{\bullet, m}_{q}(\emptyset) \right).
\end{align*}
If $\fs^{\rT} = \emptyset$ then we have

\begin{align}\label{E:Appendix EQ3}
\sBC^{\bullet, m - 1}_{q}(R_{1})
= \left( [\fs], \ - L_{1}^{m} [\fs^{\rT}, \alpha^{\bullet, m}_{q}(\emptyset)] \right),
\end{align}
and if $\fs^{\rT} \neq \emptyset$ then by Remark~\ref{remark-D-vanish} and  $[\fs^{\rT}_{+},0]=0\in \cH$, we have

\begin{align}\label{E:Appendix EQ4}
\sB^{\bullet}_{\fs^{\rT}}(\sBC^{\bullet, m - 1}_{q}(R_{1}))
= \left( [\fs] - L_{1}^{m} [\fs^{\rT}, \alpha^{\bullet, m}_{q}(\emptyset)] - L_{1}^{m} (\fs^{\rT} \boxplus \alpha^{\bullet, m}_{q}(\emptyset)), \ 0 \right).
\end{align}

Recall from~\eqref{E:R1} that $R_{1}\in \cP_{q}^{\bullet}$. It follows by Theorem~\ref{prop-BC}  that the left hand side (LHS) of each of the equations~\eqref{E:Appendix EQ1}, \eqref{E:Appendix EQ2}, \eqref{E:Appendix EQ3} and \eqref{E:Appendix EQ4} belongs to $\cP_{\wt(\fs)}^{\bullet}$. Thus, in any case of the four equations above we have that
\begin{align*}
0 = \sLB(LHS) = \sLB(RHS) = \sLB(\fs) - \sLB(\sUB(\fs)),
\end{align*}
whence we obtain
\[ \sLB(\fs)=\sLB(\sUB(\fs)) .\]
\end{proof}

We set
\begin{align*}
\sUB_{1}(\fs) &:= \left\{ \begin{array}{ll} - [\fs^{\rT}, q^{\{m + 1\}}, \fs''] + L_{1}^{m + 1} [\fs^{\rT}, \alpha_{q}^{\bullet, m + 1}(\fs'')] & (\fs' \neq \emptyset) \\ L_{1}^{m} [\fs^{\rT}, \alpha_{q}^{\bullet, m}(\emptyset)] & (\fs' = \emptyset) \end{array} \right., \\
\sUB_{2}(\fs) &:= \left\{ \begin{array}{ll} L_{1}^{m + 1} (\fs^{\rT} \boxplus \alpha_{q}^{\bullet, m + 1}(\fs'')) & (\fs' \neq \emptyset) \\ L_{1}^{m} (\fs^{\rT} \boxplus \alpha_{q}^{\bullet, m}(\emptyset)) & (\fs' = \emptyset) \end{array} \right., \\
\sUB_{3}(\fs) &:= \left\{ \begin{array}{ll} - [\fs^{\rT}, q^{\{m\}}, D^{\bullet}_{q, \fs''}] & (\fs' \neq \emptyset) \\ 0 & (\fs' = \emptyset) \end{array} \right..
\end{align*}
It is clear that $\sUB_{i}(\fs) \in \cH_{\wt(\fs)}$ for $1 \leq i \leq 3$ and $\sUB(\fs) = \sUB_{1}(\fs) + \sUB_{2}(\fs) + \sUB_{3}(\fs)$.

\begin{lemma} \label{lemma-123}
Let $\fs \in \cI \setminus \IT$.
Then the following hold:
\begin{enumerate}
\item If $\fn \in \Supp(\sUB_{1}(\fs))$, then $\dep(\fn) > \dep(\fs)$.
\item If $\fn \in \Supp(\sUB_{2}(\fs))$, then $\dep(\fn) \geq \dep(\fs)$ and $\Init(\fn) > \Init(\fs)$ (using lexicographical order).
\item If $\fn \in \Supp(\sUB_{3}(\fs))$, then $\dep(\fn) \geq \dep(\fs) > \ell_{1} := \dep(\Init(\fs))$, \ $\Init(\fn) \geq \Init(\fs)$ and $1 \leq n_{\ell_{1} + 1} < s_{\ell_{1} + 1}$.
\end{enumerate}
\end{lemma}

\begin{proof}
First we note that since $\fs \notin \IT$, if $\fs' = \emptyset$ then $m \geq 1$.
By Corollary \ref{cor-depth}, we have $\dep(\fn) \geq \dep(\fs)$ for all $1 \leq i \leq 3$ and $\fn \in \Supp(\sUB_{i}(\fs))$ and $\dep(\fn) > \dep(\fs)$ for all $\fn \in \Supp(\sUB_{1}(\fs))$.

By the definition of $\alpha_{q}^{\bullet}$, when $\fn \in \sUB_{2}(\fs)$ we can write $\fn = \fs^{\rT} \boxplus [1, \fn_{0}]$ for some $\fn_{0} \in \cI$. We may assume $\fs^{\rT}\neq \emptyset$ since if $\fs^{\rT}=\emptyset$, we have $\sUB_{2}(\fs)=\fs^{\rT}\boxplus \alpha_{q}^{\bullet, \ell}(\fs'')=0\in \cH$. Let $j:=\dep \fs^{\rT}$ and so $s_{j}<q$.
%Let $\fn' \in \Supp(\alpha^{\bullet, m + 1}_{q}(\fs''))$ (resp.\ $\fn' \in \Supp(\alpha^{\bullet, m}_{q}(\emptyset))$) when $\fs' \neq \emptyset$ (resp.\ $\fs' = \emptyset$).
Since $s_{j} + 1 \leq q$, we have $\Init(\fn) = (s_{1}, \ldots, s_{j - 1}, s_{j} + 1, \ldots) > (\fs^{\rT}, q^{\{m\}}) = \Init(\fs)$.

Let $\fn = (n_{1}, \ldots) \in \Supp(\sUB_{3}(\fs))$.
We may assume that $\fs' \neq \emptyset$ and $\bullet = \zeta$. In paritcular, $\dep(\fs) > \ell_{1}$.  Note that any $\fn^{\dagger} \in \Supp(D^{\zeta}_{q, \fs''})$ can be written as $\fn^{\dagger} = (s_{\ell_{1} + 1} - i, \fn^{\dagger}_{-})$ for some $1 \leq i < s_{\ell_{1} + 1}$ and $\fn^{\dagger}_{-} \in \cI$. Then by the definition of $\sUB_3(\fs)$, we have 
$\fn = (\fs^{\rT}, q^{\{m\}}, \fn^{\dagger}) = (\fs^{\rT}, q^{\{m\}}, s_{\ell_{1} + 1} - i, \fn^{\dagger}_{-})$.
Thus $\Init(\fn) = (\fs^{\rT}, q^{\{m\}}, \ldots) \geq (\fs^{\rT}, q^{\{m\}}) = \Init(\fs)$ and $n_{\ell_{1} + 1} = s_{\ell_{1} + 1} - i < s_{\ell_{1} + 1}$.
\end{proof}

\begin{theorem}[Theorem \ref{theorem-algo} (2)] \label{theorem-sUe}
For each $P \in \cH$, there exists $e \geq 0$ such that $\Supp((\sU^{\bullet})^{e}(P)) \subset \IT$, where $(\sU^{\bullet})^{0}$ is defined to be the identity map and for $e\in\ZZ_{>0}$, $(\sU^{\bullet})^{e}$ is defined to be the $e$-th iteration of $\sU^{\bullet}$.
\end{theorem}

\begin{proof}
We may assume that $P = \fs = (s_{1}, \ldots) \in \cI_{w}$ for some $w \geq 0$.
By Lemma \ref{lemma-123}, for each index $\fn = (n_{1}, \ldots) \in \Supp(\sUB(\fs))$ one of the following conditions holds:
\begin{enumerate}
\item $\fs \in \ITw$ (and hence $\fn = \fs$);
\item $\fs \notin \ITw$, $\dep(\fn) > \dep(\fs)$;
\item $\fs \notin \ITw$, $\dep(\fn) = \dep(\fs)$, \ $\Init(\fn) > \Init(\fs)$ (using lexicographical order);
\item $\fs \notin \ITw$, $\dep(\fn) = \dep(\fs)$, \ $\Init(\fn) = \Init(\fs)$, \ $n_{\dep(\Init(\fn)) + 1} < s_{\dep(\Init(\fs)) + 1}$.
\end{enumerate}
This means that $\fn = \fs \in \ITw$ or
\begin{align*}
(\dep(\fn); \Init(\fn); - n_{\dep(\Init(\fn)) + 1}) > (\dep(\fs); \Init(\fs); - s_{\dep(\Init(\fs)) + 1})
\end{align*}
in $\{ 0, 1, \ldots, w \} \times \Init(\cI_{w}) \times \{ - w, \ldots, - 2, - 1 \}$ with the lexicographical order.
Here, when $\Init(\fs) = \fs$ (resp.\ $\Init(\fn) = \fn$), temporary we put $s_{\dep(\Init(\fs)) + 1} := 1$ (resp.\ $n_{\dep(\Init(\fn)) + 1} := 1$).
Since this totally ordered set is finite, this procedure will stop after a finite number of steps.
\end{proof}

\end{document}